\documentclass[a4paper,12pt]{article}

\usepackage{amsmath,amsfonts,amssymb,amsthm,amscd}

\newtheorem{thm}{Theorem}
\newtheorem{lem}[thm]{Lemma}

\newtheorem{defn}[thm]{Definition}

\newtheorem{rem}[thm]{Remark}
\newtheorem{notation}[thm]{Notation}

\title{$(T)$-structures  over $2$-dimensional  $F$-manifolds: formal classification}

\date{}

\author{Liana David and Claus Hertling}

\begin{document}

\maketitle

{\bf Author's addresses}:  {\it Liana David}, ``Simion Stoilow'' Institute of Mathematics of the Romanian Academy, Calea Grivitei 21, Sector 1, Bucharest, Romania; e-mail: : liana.david@imar.ro . tel. no: 0040213196506;\

{\it Claus Hertling},  Lehrstuhl f\"{u}r Mathematik, Universit\"at Mannheim, B6, 26, 68131, Mannheim, Germany;
e-mail:  hertling@math.uni-mannheim.de\\

{\bf Abstract:}  A  $(TE)$-structure $\nabla$  over a complex manifold $M$ is a meromorphic connection 
defined on a holomorphic vector  bundle over $\mathbb{C}\times M$, 
with  poles of Poincar\'e rank one along $\{ 0 \} \times M.$ 
Under a mild additional condition (the so called unfolding condition),  
$\nabla$  induces a multiplication on $TM$ and a vector field on $M$ (the Euler field), which make $M$ into  
an  $F$-manifold with Euler field. 
By taking the pull-backs of $\nabla$ under the inclusions $\{ z\} \times M \rightarrow \mathbb{C}\times M$ we obtain 
 a family of flat connections on vector bundles over $M$,  parameterized 
by $z\in \mathbb{C}^{*}$. The properties of such a family of connections give rise to the notion of  $(T)$-structure. 
Therefore, any $(TE)$-structure underlies a $(T)$-structure but the converse is not true. The unfolding condition
can be defined also for $(T)$-structures.
A $(T)$-structure with the unfolding condition induces on its parameter space the structure of an $F$-manifold (without Euler field). 
After a brief review on the theory of $(T)$ and $(TE)$-structures, 
we determine normal forms for the equivalence classes, under formal isomorphisms, of $(T)$-structures 
which induce a given irreducible germ of $2$-dimensional $F$-manifolds. \\



{\bf Acknowledgements:}   L.D. was supported by a grant of the  Ministry of Research and Innovation, 
CNCS-UEFISCDI,  project no.  PN-III-P4-ID-PCE-2016-0019 within PNCDI III.  
Part of this work was done during her visit
at University of Mannheim (Germany) in October 2017. She thanks  University of Mannheim for hospitality and 
great working conditions.

\section{Introduction}

The theory of meromorphic connections is a well-established field with importance in many areas of modern
mathematics (complex analysis, algebraic geometry, differential geometry, integrable systems etc).
An important class of meromorphic connections are the so called $(TE)$-structures.  
They are meromorphic connections defined on holomorphic vector bundles over  products $\mathbb{C}\times  M$,
with poles of Poincar\'e rank one along the submanifold $\{ 0\}\times M$. They represent  the simplest class of meromorphic
connections with  irregular singularities along $\{ 0\}\times M.$  
The parameter space  $M$ of a $(TE)$-structure inherits, under a mild additional condition
(the  'unfolding condition') 
a multiplication $\circ$ on $TM$, with nice properties
(fiber-preserving, commutative, associative, with unit field, and satisfying a certain integrability condition), and a vector field $E$ which rescales $\circ$, 
making $M$ into a so called $F$-manifold with  Euler field. 
The notion of an $F$-manifold was  introduced for the first time in \cite{HM} as a  generalization of the notion of a Frobenius manifold \cite{Dub}. Any Frobenius manifold without metric is an
$F$-manifold. As shown in \cite{HMT}, there are $F$-manifolds which cannot be enriched to a Frobenius manifold.  Examples of $F$-manifolds arise also  in the theory of integrable systems \cite{LP, Str} and quantum cohomology \cite{HMT}.

A natural question which arises in this context  is to classify  
the  $(TE)$-structures  over a given germ of  $F$-manifolds with Euler field. 
While a $(TE)$-structure $\nabla$ may be seen as a family of meromorphic connections on 
vector bundles over $\Delta$ (a small disc centred at the origin $0\in \mathbb{C}$),  by 'forgetting' the derivatives  $\nabla_{X}$
where $X\in {\mathcal T}_{M}$ is lifted naturally to $\mathbb{C}\times M$  (this point of view being crucial in the theory of isomondromic
deformations), we may adopt the alternative view-point and study the derivatives $\nabla_{X}$ 
as a family of flat connections on vector bundles over $M$ parameterised by $z\in \mathbb{C}^{*}$.  
Such a family has received much attention in the theory of meromorphic connections and is referred  in the literature as a $(T)$-structure over $M$. 
Therefore, any $(TE)$-structure underlies a $(T)$-structure but the converse is not always true. 
The parameter space of a $(T)$-structure inherits the structure of an $F$-manifold (without Euler field), when the unfolding condition is satisfied.

Adopting the second view-point, in this paper we  make a first step in the classification of $(TE)$-structures over a given germ of $F$-manifolds with Euler field. 
We consider the simplest case, namely when the germ  is $2$-dimensional and  irreducible and we determine formal  normal forms 
for the $(T)$-structures over such germs. The results we prove here will be crucial for future projects, where we shall classify $(TE)$-structures
over $2$-dimensional (and, possibly bigger dimensional) germs of $F$-manifolds with Euler fields. 
The $2$-dimensional case is considerably simpler, owing to the fact that  (unlike  higher dimensions)
irreducible germs of   $2$-dimensional $F$-manifolds are classified \cite{Hbook}: either they coincide with the germ of the globally nilpotent constant $F$-manifold $\mathcal N_{2}$
or they are generically semisimple and belong to a class of germs $I_{2}(m)$ parameterized by  $m\in \mathbb{N}_{\geq 3}$
(see the end of Section \ref{s2.2} for the description of these germs). 
As $F$-manifolds isomorphisms lift to isomorphisms between the spaces of  $(T)$  and $(TE)$-structures  over them, 
we can (and will) assume, without loss of generality,  that our germs of $F$-manifolds  are  $\mathcal N_{2}$ or $I_{2}(m)$ (with $m\geq 3$).  
The specific form of 
these germs   will   enable us to find  the formal normal forms for $(T)$-structures  over them. \\

{\bf Structure of the paper.}  
In Section \ref{preliminary} we recall   well-known facts we need on $(T)$,  $(TE)$-structures  and $F$-manifolds (see e.g. \cite{He03}).
Although our original contribution in this paper refers to $(T)$-structures, we include also basic material on $(TE)$-structures as a motivation and to fix
notation and results we shall use  in the subsequent stages of our project on classification of $(TE)$-structures. 
In Section \ref{diff-sect} we study various classes of differential equations which will be relevant in our treatment.
In  Section \ref{formal-i2} we determine the formal normal forms for  $(T)$-structures over $I_{2}(m)$ 
and in Section \ref{formal-n2}  we study the similar question for 
$(T)$-structures over $\mathcal N_{2}$.
A main difference between these two cases lies in the form of formal isomorphisms used in the classification. 
The automorphism group of $I_{2}(m)$ is finite (see Lemma \ref{aut-F-man})
and  formal  $(T)$-structure isomorphisms which
lift  non-trivial automorphisms of $I_{2}(m)$ do not add much simplification in
the expressions of  $(T)$-structures over $I_{2}(m)$.  
For this reason, in the case of $I_{2}(m)$ we content ourselves to  
determine formal normal forms  for $(T)$-structures which are formally gauge isomorphic
(i.e. are isomorphic by means of formal isomorphisms which lift the identity map of $I_{2}(m)$). 
This is done in  Theorem \ref{t4.1}. 
As opposed to $I_{2}(m)$, the germ $\mathcal N_{2}$ has a rich automorphism group
(see Lemma \ref{aut-F-man}). The $(T)$-structures over $\mathcal N_{2}$ will be classified up to
formal gauge isomorphisms in   
Theorem \ref{t5.1} and up to the entire group of formal isomorphisms in 
Theorem \ref{t5.3}. Formal isomorphisms which lift non-trivial
automorphisms of $\mathcal N_{2}$  simplify considerably the classification. Their role in the classification  is explained
in Theorem \ref{t5.2}.

\section{Preliminary material}\label{preliminary}

We begin by fixing our  notation.

\begin{notation}{\rm For a complex manifold $M$, we denote by $\mathcal O_{M}$, $\mathcal T_{M}$, $\Omega^{k}_{M}$  the sheaves of holomorphic functions, 
holomorphic vector fields and holomorphic $k$-forms on $M$  respectively. 
For a holomorphic vector bundle $H$, we denote by ${\mathcal O}(H)$ the sheaf of its holomorphic 
sections. We denote by  $\Omega^{1}_{\mathbb{C}\times M}(\mathrm{log} \{ 0\} \times M)$ the sheaf of meromorphic $1$-forms on $\mathbb{C}\times M$,
which are logarithmic along $\{  0\} \times M.$
Locally, in a neighborhood of $(0, p)$, where $p\in M$, any $\omega \in \Omega^{1}_{\mathbb{C}\times M}(\mathrm{log} \{ 0\} \times M)$ is of the form
$$
\omega = \frac{f(z,t)}{z} dz + \sum_{i} f_{i}(z,t) dt_{i}
$$
where $(t_{i})$ is a coordinate system of $M$ around $p$  and  $f$, $f_{i}$ are holomorphic.   
The ring of holomorphic functions defined on a neighbourhood of $0\in \mathbb{C}$ will be denoted  by 
$\mathbb{C}\{ z\}$, the ring of formal power series $\sum_{n\geq 0} a_{n} z^{n}$ will be denoted by
$\mathbb{C}[[z]]$, the subring of power series  $\sum_{n\geq 0} a_{n} z^{n}$ with $a_{n} =0$ for any
$n\leq k-1$ will be denoted by $\mathbb{C} [[z]]_{\geq k}$ 
and the vector space of polynomials of degree at most $k$ in the variables $t=(t_{i})$ 
will be denoted by $\mathbb{C}[t]_{\leq k}.$  
Finally, we  denote by 
$\mathbb{C}\{ t, z]]$ the ring of 
formal power series $\sum_{n\geq 0} a_{n} z^{n}$ where all 
$a_{n}= a_{n}(t)$ are holomorphic on the {\it same} neighbourhood of $0\in \mathbb{C}$
and by $\mathbb{C}[[z]][t]_{\leq k}$ the vector space
of power series  $\sum_{n\geq 0} a_{n} z^{n}$ with $a_{n}$  polynomials of degree at most $k$  in $t$. 
For a  function $f\in\mathbb{C}\{t,z]]$
and matrix $A\in M_{k\times k}(\mathbb{C}\{t,z]])$,
we often write $f=\sum_{n\geq 0}f(n)z^n$ and 
$A=\sum_{n\geq 0}A(n)z^n$ where  $f(n)\in\mathbb{C}\{t\}$
and $A(n)\in M_{k\times k}(\mathbb{C}\{t\})$. }
\end{notation}

\subsection{$(T)$ and $(TE)$-structures}\label{t-te}

In this section we recall basic facts on $(T)$ and  $(TE)$-structures. 

\begin{defn}\label{t1.1} Let $M$ be a complex manifold and $H\rightarrow  \mathbb{C}\times M$
a holomorphic vector bundle.

i) \cite[Def. 3.1]{HM2}
A $(T)$-structure over $M$ is a pair $(H\rightarrow  \mathbb{C}\times M,\nabla )$
where $\nabla$ is a map
\begin{equation}\label{1.1}
\nabla:{\mathcal O}(H)\to \frac{1}{z}\mathcal O_{\mathbb{C}\times M}\cdot\Omega^1_M\otimes 
{\mathcal O}(H)
\end{equation}
such that, for any $z\in\mathbb{C}^*$,  the restriction of $\nabla$
to $H\vert_{\{z\}\times M}$ is a flat connection.

ii) \cite[Def. 2.1]{HM2}
A $(TE)$-structure over $M$ is a pair $(H\rightarrow \mathbb{C}\times M,\nabla )$
where $\nabla $ is a flat connection on $H\vert_{\mathbb{C}^{*}\times M}$
with  poles of Poincar\'e rank 1 along $\{0\}\times M$:
\begin{equation}\label{1.2}
\nabla:{\mathcal O}(H)\to \frac{1}{z}\Omega^1_{\mathbb{C}\times M}
(\mathrm{log}(\{0\}\times M)\otimes {\mathcal O}(H).
\end{equation}
\end{defn}

Any $(TE)$-structure determines (by forgetting the derivative in the $z$ direction) a $(T)$-structure
('E' comes from extension).\

Let $(H\rightarrow \mathbb{C}\times M,\nabla )$ be a $(TE)$-structure and 
$\Delta \subset \mathbb{C}$  a small disc centred at the origin, 
$U\subset M$ a coordinate chart with coordinates $(t_{1}, \cdots , t_{m})$,  such that  
$H\vert_{\Delta \times U}$ is trivial.  Using a trivialization 
$\underline{s}=(s_1,\cdots ,s_r)$ of 
$H\vert_{\Delta\times U}$, we write 
\begin{eqnarray}\label{2.8}
\nabla (s_{i})&=&\sum_{j=1}^{r} \Omega_{ji} s_{j},\quad  
\textup{short}:\quad 
\nabla(\underline{s})=\underline{s}\cdot \Omega, \\
\Omega &=& \sum_{i=1}^{m }  z^{-1} A_{i}(z,t) dt_{i}
+z^{-2}B(z,t)dz,\nonumber
\end{eqnarray}
where $A_{i}$, $B$ are holomorphic, 
\begin{equation}\label{2.9}
A_i(z,t)=\sum_{k\geq 0}A_i(k)z^k,\  B(z,t)=\sum_{k\geq 0} B(k)z^k
\end{equation}
and $A_{i}(k)$ and $B(k)$ depend only on $t\in U.$  
The flatness of 
$\nabla$ gives,  for any $i\neq j$,
\begin{align}
\label{2.10}&   z\partial_iA_j-z\partial_jA_i+[A_i,A_j]=0,\\
\label{2.11}&   z\partial_i B-z^2\partial_z A_i + zA_i + [A_i,B] =0.
\end{align}
(When $\nabla$ is a $(T)$-structure, the summand $z^{-2}B(t,z)d z$
in $\Omega$ and relations (\ref{2.11}) are dropped).
Relations  (\ref{2.10}),  (\ref{2.11}) 
split according to the  powers
of $z$  as follows: for any $k\geq 0$, 
\begin{align}
\label{2.12}&  \partial_iA_j(k-1)-\partial_jA_i(k-1)+\sum_{l=0}^k[A_i(l),A_j(k-l)] =0,\\
\label{2.13} &\partial_i B(k-1)-(k-2)A_i(k-1) + \sum_{l=0}^k[A_i(l),B(k-l)]= 0,
\end{align}
 where  $A_i(-1)=B(-1)=0$.

\begin{defn}\label{def-t-te}  i) An  isomorphism 
$T:(\tilde{H},\tilde{\nabla}) \rightarrow(H,\nabla)$
between two $(T)$-structures over $\tilde{M}$ and $M$
respectively is a holomorphic  vector bundle isomorphism
$T: \tilde{H}\rightarrow H$ 
which covers a biholomorphic map of the form
$\mathrm{Id} \times h : \mathbb{C}\times \tilde{M} \rightarrow \mathbb{C} \times M$, i.e. 
$T(\tilde{H}_{(z,\tilde{p})})\subset H_{(z,h(\tilde{p}))}$,  
for any $\tilde{p}\in \tilde{M}$, 
and  is compatible with connections:
\begin{equation}\label{compat-t-gl}
T (\tilde{\nabla}_{X_{\tilde{p}}}(s) ) = \nabla_{h_{*}(X_{\tilde{p}})} (T(s)),\ \forall X_{\tilde{p}}\in T_{\tilde{p}}\tilde{M},\ \tilde{p}\in \tilde{M}, s\in
{\mathcal O }(\tilde{H}).
\end{equation}
Above, $T(s) \in {\mathcal O }(H)$ is defined by
$T(s)_{(z, p)} := T ( s_{(z,h^{-1}(p))})$, for any $p\in M.$\

ii) An isomorphism $T:(\tilde{H},\tilde{\nabla}) \rightarrow (H,\nabla)$
between two $(TE)$-structures   
is an isomorphism between their  underlying $(T)$-structures,  
which satisfies in addition 
\begin{equation}\label{compat-te-gl}
T (\tilde{\nabla}_{\partial_{z}}(s) ) = \nabla_{\partial_{z}} (T(s)),\ \forall 
s\in {\mathcal O }(\tilde{H}).
\end{equation}
\end{defn}

Recall that if $f: \tilde{N}\rightarrow N$ is a map and 
$\pi : E\rightarrow N$ is
a bundle over $N$ then $f^{*}E:= 
\{ (e,\tilde{p})\in E\times \tilde{N},\ \pi(e) = f(\tilde{p})\}$ is a bundle over $\tilde{N}$ with bundle projection $(e,\tilde{p})\rightarrow \tilde{p}.$ 
Any 
section $s \in {\mathcal O}(E)$   
defines a section $f^{*}s\in {\mathcal O }( f^{*}E)$ by 
$(f^{*}s)(\tilde{p}):= (s_{f(\tilde{p})},\tilde{p}).$ 
If $f$ is a biholomorphic map, then
there is a natural bundle isomorphism $f^{*} : E\rightarrow f^{*}E$ 
which covers $f^{-1}$.  
Finally, if $\nabla$
is a connection on $E$, then the pull-back connection 
$f^{*}\nabla$ on $f^{*} E$ 
is defined by $(f^{*}\nabla)_{X_{\tilde{p}}}(f^{*}s) := 
f^{*} (\nabla_{f_{*}(X_{\tilde{p}})} (s))$, for any 
$X_{\tilde{p}}\in T_{\tilde{p}}\tilde{N}$,  
$\tilde{p}\in \tilde{N}$
and $s\in {\mathcal O}(E).$

\begin{notation}{\rm 
For simplicity,
we will say that a $(T)$ or $(TE)$-structure isomorphism as in Definition 
\ref{def-t-te} covers $h$ instead of $\mathrm{Id}\times h$ and write $h^{*}$ for the pull-back
 $(\mathrm{Id}\times h)^{*}$ 
(of bundles, connections, etc).
Similarly, we will sometimes write $A_{i} \circ h$ instead of $A_{i}\circ (\mathrm{Id} \times h)$ 
and  $T\circ h$ instead of $T\circ (\mathrm{Id}\times h)$. }
\end{notation}

The next lemma can be checked directly.

\begin{lem}\label{change-base} 
Let $(\tilde{H},\tilde{\nabla})$ and $(H,\nabla)$ be two 
$(T)$-structures over $\tilde{M}$ and $M$ respectively. 
If $T: (\tilde{H},\tilde{\nabla})\rightarrow (H,\nabla)$
is an isomorphism which covers $h: 
\tilde{M}\rightarrow M$, then 
$h^{*}\circ T : (\tilde{H}, \tilde{\nabla}) \rightarrow
(h^{*}H, h^{*}\nabla)$ 
is an isomorphism which covers the identity map of $\tilde{M}$.
\end{lem}

Let  $T:(\tilde{H},\tilde{\nabla}) \rightarrow (H,\nabla)$
be a  $(T)$ or $(TE)$-structure isomorphism 
over $\tilde{M}$ and $M$ respectively, which covers a biholomorphic map 
$h:\tilde{M} \rightarrow M.$ 
Fix trivializations 
$\underline{\tilde{s}}=(\tilde{s}_{1}, \dots , \tilde{s}_{r})$ and 
$\underline{s}=(s_{1},\dots ,s_{r})$  
of $\tilde{H}$ and $H$ over $\Delta \times \tilde{U}$ and 
$\Delta \times U$
respectively, where 
$\tilde{U}\subset \tilde{M}$,  $U\subset M$ 
are open subsets and $U = h(\tilde{U}).$ 
Then the isomorphism $T$ is given by a matrix
$(T_{ij}) = \sum_{k\geq 0}T(k) z^k
\in M(r\times r,\mathcal O _{\Delta \times U})$ with 
$T(k)\in M(r\times r,\mathcal O_{U})$, $T(0)$ invertible, such that 
$T(\tilde{s}_{i}) = \sum_{j=1}^{r} T_{ji} s_{j}$, 
or, explicitly, 
\begin{eqnarray}\label{t-s-formal}
T((\tilde{s}_{i})_{(z,\tilde{p})} ) &=& \sum_{j=1}^{r} T_{ji}(z,  h(\tilde{p}) ) (s_{j})_{(z, h(\tilde{p}))},\ \forall \tilde{p}\in \tilde{U},\ z\in \Delta.
\end{eqnarray}
We write relation (\ref{t-s-formal}) shortly as
$$
T((\underline{\tilde{s}})_{(z, \tilde{p})}) = 
(\underline{s})_{(z, h(\tilde{p}))} \cdot T (z, h(\tilde{p})).
$$
Suppose  now that  $(\tilde{t}_{1}, \cdots , \tilde{t}_{m})$ and $(t_{1}, \cdots , t_{m})$ are local coordinates of
$\tilde{M}$ and $M$, defined on $\tilde{U}$ and $U$ respectively.  The compatibility relations
(\ref{compat-t-gl}),  (\ref{compat-te-gl})  read 
 \begin{align}
\label{2.15} &z \partial_{i}(\tilde{T}) + \sum_{j=1}^{m} (\partial_{i}{h^{j}})( A_{j} \circ h ) \tilde{T}- \tilde{T} 
\tilde{A}_{i}=0,\ \forall i\\
\label{2.16}& z^{2}\partial_{z}(\tilde{T})+ (B\circ h) \tilde{T} - \tilde{T}\tilde{B} =0,
\end{align}
where $\tilde{T}:= T \circ h$ and $(h^{j})$ are the components of  the representation of $h$ in the two charts
(relation (\ref{2.16}) has to be omitted when $\tilde{\nabla}$ and $\nabla$ are $(T)$-structures). Relations
(\ref{2.15}),  (\ref{2.16})  split
according to the powers of $z$ as
\begin{align}
\label{compat-t}&  \partial_i \tilde{T}(k-1)+ \sum_{l=0}^k (\sum_{j=1}^{m} (\partial_{i}{h^{j}}) (A_j(l)\circ h)  \tilde{T}(k-l)
-\tilde{T}(k-l) \tilde{A}_i(l))=0\\
\label{compat-te}&  (k-1)\tilde{T}(k-1)+\sum_{l=0}^k( (B(l)\circ h )  \tilde{T}(k-l)
-\tilde{T}(k-l)  \tilde{B}(l)) =0,
\end{align}
for any $k\geq 0$, where $\tilde{T}(-1)=0.$ 

We now discuss a particular class of $(T)$ and $(TE)$-structure isomorphisms, called  gauge isomorphisms.
Consider a $(T)$ or a $(TE)$-structure 
$(H, \nabla )$ with $H$ trivial, and  
$\underline{s}=(s_{1},\dots,s_{r})$ a trivialization of $H$. Any other trivialization 
$\underline{\tilde{s}}=(\tilde{s}_{1},\dots, \tilde{s}_{r})$ 
of $H$ is related to $\underline{s}$
by an invertible holomorphic matrix-valued function  $T = (T_{ij})$ defined by
$\tilde{s}_{i} = \sum_{j=1}^{r} T_{ji} s_{j}$
(short: 
$\underline{\tilde{s}}=\underline{s}\cdot T$). 
Suppose that  the connection form  $\Omega$ of $\nabla$ in the trivialization  $\underline{s}$ is given by 
(\ref{2.8}), (\ref{2.9})  (without the term  $B$, when $\nabla$ is a $(T)$-structure). 
Then the connection form $\tilde{\Omega}$ of $\nabla$ in the new trivialization  $\underline{\tilde{s}}$
has the same form,  with matrices $\tilde{A}_{i}$ and $\tilde{B}$ related to $A_{i}$ and $B$ by
\begin{align}
\label{2.17-w}z\partial_{i}(T) + A_{i}T - T\tilde{A}_{i}=0\\
\label{2.18-w}z^{2} \partial_{z}(T) + BT - T\tilde{B}=0, 
\end{align}
or by 
\begin{align}
\label{2.17} \partial_i T(k-1)+\sum_{l=0}^k (A_i(l)  T(k-l)
-T(k-l) \tilde{A}_i(l))=0,\\
\label{2.18} (k-1){T}(k-1)+\sum_{l=0}^k(B(l)  T(k-l)
-T(k-l)  \tilde{B}(l))=0.
\end{align}
for any $k\geq 0.$  We say that $T$ defines  a gauge isomorphism  between the $(T)$ (or $(TE)$-structures) with connection forms $\Omega$ and $\tilde{\Omega }$.

\begin{rem}{\rm  i) 
Gauge isomorphisms  are isomorphisms between $(T)$ or $(TE)$-structures
over the same base $M$ which lift the identity map of $M$.
Remark that relations (\ref{compat-t}),  
(\ref{compat-te}) with $h=1$ reduce to 
(\ref{2.17}), (\ref{2.18}).\  

ii) $(T)$ and $(TE)$-structures can be defined also over germs of manifolds. In this case we always assume
that their underlying bundles are trivial.  Isomorphisms
between them which lift a given biholomorphic map   of   their parameter spaces  
are given simply by matrices,  as explained above (the matrix 
$T\in M(r\times r, {\mathcal O}_{\Delta \times U})$ 
in the above notation).\

iii)    
$(T)$ and $(TE)$-structures   can  be extended to the formal setting as follows. 
A formal $(T)$ or $(TE)$-structure over a germ $(M, 0)$ 
is a pair $(H, \nabla )$, where $H\rightarrow (\mathbb{C}, 0) \times (M,0)$  
is the germ of a holomorphic
vector bundle and $\nabla$ is given by a connection form   (\ref{2.8}), where $A_{i}$ and $B$
(the latter only when $\nabla$ is a $(TE)$-structure) are matrices  with entries in 
$\mathbb{C} \{ t, z]]$, satisfying relations (\ref{2.10}), (\ref{2.11}) or (\ref{2.12}), (\ref{2.13}) (relations
 (\ref{2.11}), (\ref{2.13}) only when $\nabla$ is a $(TE)$-structure). 
A formal isomorphism between two formal $(T)$ or $(TE)$-structures $(H, \nabla )$ and $(\tilde{H}, \tilde{\nabla})$ over
$(M,0)$ and $(\tilde{M}, 0)$ respectively, which covers a biholomorphic map 
$h: (\tilde{M}, 0) \rightarrow ({M}, 0)$,  is given by a matrix  $T= (T_{ij})$
with entries $T_{ij}\in \mathbb{C} \{ t, z]]$,  such that  relations  (\ref{2.15}),  (\ref{2.16}) 
or (\ref{compat-t}), (\ref{compat-te}) 
are satisfied with $\tilde{T} = T\circ h$  
(relations (\ref{2.16}),  (\ref{compat-te}) only  when $\nabla$ and $\tilde{\nabla}$ are  $(TE)$-structures). Formal gauge isomorphisms between $(T)$ or $(TE)$-structures over the same germ $(M,0)$
are formal isomorphisms which cover the identity map of $(M,0)$. They are given by matrices  $T= (T_{ij})$ with entries in 
$\mathbb{C} \{ t, z]]$ such that relations 
(\ref{2.17-w}), (\ref{2.18-w}) or 
(\ref{2.17}), (\ref{2.18}) are satisfied.  }
\end{rem}

\subsection{$(T)$-structures and  $F$-manifolds}\label{s2.2}

\subsubsection{General results}\label{gen-sect}

Let $(H, \nabla )$ be a $(T)$-structure over a manifold $M$.  It induces a Higgs field $C\in \Omega^{1}(M, \mathrm{End}(K))$
on the restriction $K:=H_{|\{0\}\times M}$, defined by
\begin{equation}
C_X[a]:=[z\nabla_Xa],\ \forall  X\in {\mathcal T}_M,a\in{\mathcal O} (H),
\end{equation}
where $[\ ]$ means the restriction to $\{ 0\} \times M$ and $X\in{\mathcal T}_M$
is lifted canonically from its domain of definition $U\subset M$
to ${\mathbb C}\times U$.  
In the notation from Section \ref{t-te},  $C$ is given locally by $\sum_{i=1}^{m}A_{i}(0)dt_{i}.$ 
Relation (\ref{2.12}) with $k=0$ implies
$[ A_{i}(0), A_{j}(0)]=0$, i.e. $C_XC_Y=C_YC_X$ for any $X, Y\in TM$, which we  write  as  $C\wedge  C=0$.
We say that   $C$ is a Higgs field and $(K, C)$ is a Higgs bundle.

If $(H, \nabla )$ is a $(TE)$-structure then there is in addition an endomorphism $\mathcal U\in \mathrm{End}(K) $, 
\begin{equation}
\mathcal U :=[z\nabla_{z\partial_{z}}]:\mathcal O (K)\rightarrow  \mathcal O (K).
\end{equation}
It satisfies  $C_X\mathcal U =\mathcal U  C_X$ for any $X\in TM$. We write this as  $[C,\mathcal U ]=0$.

\begin{defn}\label{t1.2} \cite[Theorem 2.5]{HM2}
(a) A Higgs bundle $(K\rightarrow M , C)$  
satisfies the {\it unfolding condition} 
if there is an open cover $\mathcal U$ of $M$ such that, on any $U\in \mathcal U$,   there is $\zeta\in {\mathcal O} (K\vert_{U})$ (called a primitive section) 
with the property that the map
$TU\ni X\rightarrow C_{X} \zeta \in K\vert_{U}$  is an isomorphism.\ 

(b) A $(T)$-structure (or a $(TE)$-structure) satisfies the 
{\it unfolding condition} if the induced Higgs bundle  
satisfies the unfolding condition.
\end{defn}

\begin{rem}{\rm 
If $(H\rightarrow \mathbb{C}\times M, \nabla )$ is a $(T)$-structure with the unfolding condition, then the rank of $H$ and the dimension of 
$M$ coincide.}
\end{rem}

We now define the parallel notion of  $F$-manifold. 

\begin{defn}\label{t2.3} \cite{HM}
A complex manifold $M$ with a fiber-preserving, commutative, associative  
multiplication $\circ$ on the holomorphic 
tangent bundle $TM$ and unit field $e\in {\mathcal T}_M$ is an
{\it $F$-manifold} if 
\begin{equation}\label{2.19}
L_{X\circ Y}(\circ) =X\circ L_Y(\circ)+Y\circ L_X(\circ),\  \forall X, Y\in {\mathcal T}_{M}.
\end{equation}
A vector field $E\in{\mathcal T}_M$ is called an {\it Euler field}
(of weight $1$)  if
\begin{equation}\label{2.20}
L_E(\circ) =\circ.
\end{equation}
\end{defn}

The following lemma was  proved in 
Theorem 3.3 of \cite{HHP10}. The proof below is more elegant and shorter.

\begin{lem}\label{t2.4}
A $(T)$-structure  $(H\rightarrow \mathbb{C}\times M, \nabla )$  
with unfolding condition 
induces a multiplication $\circ$ on $TM$ which makes $M$ an $F$-manifold.
A $(TE)$-structure $(H\rightarrow \mathbb{C}\times M, \nabla )$   with unfolding condition induces in addition
a vector field $E$ on $M$, which, together with $\circ$, makes $M$ an $F$-manifold with Euler field.
\end{lem}

\begin{proof} 
Let $(H\rightarrow \mathbb{C}\times M, \nabla )$ be a $(T)$ or $(TE)$-structure with induced Higgs bundle
$(K, C)$. We define $\circ$ by
\begin{equation}\label{higgs}
C_{X\circ Y}\zeta =  C_{X}C_{Y}\zeta ,\   \forall X, Y\in TM, 
\end{equation}
where $\zeta$ is a local primitive section. We remark that $\circ$ has unit field $e$ determined by the condition 
$C_{e}\zeta  = \zeta$.  When $\nabla$ is a $(TE)$-structure, the induced endomorphism $\mathcal U$ of $K$   
defines a unique vector field $E\in \mathcal T_{M}$ with $-C_E \zeta =\mathcal U (\zeta )$.
The definition of $\circ$ and $E$ are independent on the choice of $\zeta$  (see Lemma 4.1 of \cite{He03}).

Suppose now that $\nabla$ is a $(T)$-structure. 
In order to prove that $(M,\circ , e)$ is an $F$-manifold
it is sufficient to find
a $(1,0)$-connection $D^{\prime }$ on
the (complex) $C^\infty$-bundle underlying $TM$, with 
\begin{equation}\label{p-c-1}
D^{\prime}(C^{\prime})_{X, Y}:= D^{\prime}_{X} ( C^{\prime}_{Y}) - D^{\prime}_{Y} ( C^{\prime}_{X}) - C^{\prime}_{[X, Y]}=0,\
\forall X, Y\in {\mathcal T}_{M}   
\end{equation}
where $C^{\prime}_{X}Y := X\circ Y$.  
The sufficiency follows with Lemma 4.3 of \cite{He03}.

If $\nabla$ is a $(TE)$-structure,  
in order to prove that $(M, \circ ,e,E)$ is an $F$-manifold with Euler field  it is sufficient  to prove in addition  
the existence of  a $C^\infty$-endomorphism
$\mathcal Q^{\prime}$ of $TM$, with 
\begin{equation}\label{p-c-2}
D^{\prime}_{X}(\mathcal U^{\prime})-[C^{\prime}_{X},\mathcal Q^{\prime}]+C^{\prime}_{X}=0,\ \forall X\in {\mathcal T}_{M},
\end{equation}
where $\mathcal U^{\prime}(X) := - E\circ X$ for any $X\in TM$. The sufficieny follows again from 
Lemma 4.3 of \cite{He03}.

We will define $D^{\prime}$ and $\mathcal Q^{\prime}$ locally,
on any  open subset $U\subset M$, small enough such that there is a primitive section $\zeta$ of $K\vert_{U}$ 
and a coordinate system  $(t_1, \cdots ,t_m)$ of $M$ defined on $U$. Let 
$\partial_1,\cdots ,\partial_m$ be the associated 
coordinate vector fields.
Let $\underline{s}:= (s_1,\cdots ,s_m)$ be a trivialization of $H$ on 
$\Delta\times U$
(where $\Delta$ is a small disc centred at $0\in \mathbb{C}$) 
and let
$\Omega =\frac{1}{z} A_{i} dt_{i}+ \frac{1}{z^{2}} B dz$  be the connection form of $\nabla$ in this trivialization
(with $B=0$ when $\nabla$ is a $(T)$-structure).  
Let $\underline{s}\vert_{\{0\}\times U}=:
\underline{s}^{(0)}= (s_{1}^{(0)}, \cdots , s_{m}^{(0)})$ 
be the trivialization of $K$ obtained
by restricting $\underline{s}$ to $\{ 0\} \times U$. 
Define a $(1,0)$-connection $D$  on $K\vert_{U}$, by
\begin{equation}\label{2.21}
D_{\partial_k} ({s}^{(0)}_{i}) = \sum_{j=1}^{m} A_k(1)_{ji}s_{j}^{(0)}, 
\quad\textup{short:}
\ D_{\partial_k}(\underline{s}^{(0)})=\underline{s}^{(0)}\cdot
A_k(1),
\end{equation}
and, when $\nabla$ is  a $(TE)$-structure, an endomorphism  $\mathcal Q$  of $K\vert_{U}$  by
\begin{equation}\label{2.22}
\mathcal Q ( s_{i}^{(0)}):=-   \sum_{j=1}^{m}B(1)_{ji} s_{j}^{(0)}.
\quad\textup{short:}
\ \mathcal{Q}(\underline{s}^{(0)})=-\underline{s}^{(0)}\cdot
B(1).
\end{equation}

From relations  (\ref{2.10}) and  (\ref{2.11}), 
\begin{align}
\label{2.23} & \partial_j A_k(0)-\partial_k A_j(0)
+[A_j(1),A_k(0)]+[A_j(0),A_k(1)]=0,\\
\label{2.24}& \partial_{j} B(0)+ A_j(0)
+[A_j(1),B(0)] + [A_j(0),B(1)]=0.
\end{align}
From the definitions of  $C$ and $\mathcal U$, 
\begin{align}
C_{\partial_{j}}( s^{(0)}_{r})&=\sum_{k=1}^{m} A_{j}(0)_{kr} s^{(0)}_{k},
\quad\textup{short:}\ 
C_{\partial_{j}}(\underline{s}^{(0)})
= \underline{s}^{(0)}\cdot A_{j}(0),
\\
\mathcal U ( s_{r}^{(0)})&= \sum_{k=1}^{m}B(0)_{kr} s_{k}^{(0)},
\quad\textup{short:}\ 
\mathcal U (\underline{s}^{(0)})=
\underline{s}^{(0)} \cdot B(0).
\end{align}
Now, a straightforward computation shows that
 (\ref{2.23}) gives 
\begin{equation}\label{p-c-3}
D(C)_{X, Y}=0,\ \forall X, Y\in {\mathcal T}_{M}
\end{equation}
and   (\ref{2.24})  gives 
\begin{equation}\label{p-c-4}
D_{X}(\mathcal U)-[C_{X},\mathcal Q]+C_{X}=0,\ \forall X\in {\mathcal T}_{M}.
\end{equation}
More precisely, (\ref{p-c-3}) is obtained from the following computation: 
\begin{align*}
& D(C)_{\partial_{j}, \partial_{r}} (\underline{s}^{(0)}) = (D_{\partial_{j}} (C_{\partial r}) -   D_{\partial_{r}} (C_{\partial j}))(
\underline{s}^{(0)})\\
& = D_{\partial_{j}} ( C_{\partial r} (\underline{s}^{(0)} )) - C_{\partial_{r}} ( D_{\partial_{j}} (\underline{s}^{(0)}) ) 
-   D_{\partial_{r}} ( C_{\partial j} (\underline{s}^{(0)} ) )+  C_{\partial_{j}} ( D_{\partial_{r}} (\underline{s}^{(0)}) ) \\
& = \underline{s}^{(0)}\cdot \Bigl[
\partial_{j}( A_{r}(0)) + A_{j}(1) A_{r}(0) - A_{r}(0)A_{j}(1) \\
&- \partial_{r}( A_{j}(0)) - A_{r}(1)A_{j}(0) + A_{j}(0) A_{r}(1)\Bigr] \\
&= \underline{s}^{(0)}\cdot \Bigl[ \partial_{j} A_{r}(0) -\partial_{r} A_{j}(0) + [A_{j}(1), A_{r}(0)] + [A_{j}(0), A_{r}(1) ] \Bigr] .
\end{align*} 
 which vanishes from (\ref{2.23}). Relation 
(\ref{p-c-4})  can be proved similarly. 

By means of the isomorphism $TU\cong K\vert_{U}$ defined by 
$X \rightarrow C_{X} \zeta$, 
where $\zeta$ is a primitive section on $U$, 
the connection $D^{\prime}$ and endomorphism   
$\mathcal Q^{\prime}$
we are looking for  (the latter when $\nabla$ is a $(TE)$-structure) 
correspond to the connection $D$ and endomorphism $\mathcal Q$ respectively. 
When $\nabla$ is a $(TE)$-structure,  the endomorphisms $\mathcal U$ and $\mathcal U^{\prime}$ also   
correspond. 
Relations (\ref{p-c-1}) and  (\ref{p-c-2}) follow from (\ref{p-c-3}) and  (\ref{p-c-4}) respectively. 
\end{proof}

Let $h: (\tilde{M}, \tilde{\circ} ,\tilde{ e})\rightarrow ({M}, {\circ},  {e})$ 
be an $F$-manifold isomorphism.  
If $(E, \nabla )$ is a  $(T)$-structure over $M$ which induces $(\circ , e)$ then 
 $(h^{*}E, h^{*}\nabla )$ is a 
$(T)$-structure over $\tilde{M}$ which 
induces $(\tilde{\circ}, \tilde{e})$ (and a similar statement holds for $(TE)$-structures and $F$-manifolds with Euler fields). 
In particular, the spaces of $(T)$-structures  over  isomorphic $F$-manifolds (or isomorphic germs of $F$-manifolds) are isomorphic.
The same statement is true for the spaces of formal $(T)$-structures
over isomorphic germs of $F$-manifolds (the unfolding condition can be extended,
in the obvious way, to formal $(T)$ and $(TE)$-structures and Lemma \ref{t2.4} remains true  
when the $(T)$ or $(TE)$-structure is replaced by a formal one).

\subsubsection{Germs of $2$-dimensional $F$-manifolds}

There exist two types of  isomorphism classes of irreducible germs of $2$-dimensional  $F$-manifolds
(see \cite{Hbook}, Theorem 4.7): 
$I_{2}(m)$ with $m\in \mathbb{N}_{\geq 3}$  (generically semisimple) and  $\mathcal N_{2}$ (globally nilpotent). 
As germs of manifolds, $I_{2}(m)$ and $\mathcal N_{2}$ are $(\mathbb{C}^{2}, 0)$. In the standard coordinates
$(t_{1}, t_{2})$ of $\mathbb{C}^{2}$, the multiplication of $I_{2}(m)$ has $\partial_{1}$ as unit field and 
 $\partial_{2}\circ \partial_{2} = t_{2}^{m-2}\partial_{1}.$ Similarly, the multiplication of $\mathcal N_{2}$ has $\partial_{1}$ as unit
field and $\partial_{2}\circ \partial_{2} =0.$  The next simple lemma describes the automorphism groups of $I_{2}(m)$ and $\mathcal N_{2}.$ 

\begin{lem}\label{aut-F-man}
i) The automorphism group of  $I_{2}(m)$ ($m\geq 3$)  is cyclic of order $m$, generated by the automorphism
$$
(t_{1}, t_{2}) \rightarrow ( t_{1}, e^{\frac{2\pi i}{m}} t_{2}).
$$
ii) The automorphism group of $\mathcal N_{2}$ is the group of  all biholomorphic maps 
\begin{equation}\label{tilde-f}
(t_{1}, t_{2}) \rightarrow ( t_{1}, \lambda (t_{2})),
\end{equation}
where $\lambda\in \mathbb{C}\{ t_{2}\}$, with $\lambda (0) =0$ and $\dot{\lambda}(0) \neq 0.$ 
\end{lem}

Our aim in this paper is to find formal isomorphisms for   $(T)$-structures over  an arbitrary 
irreducible germ  $(M, 0)$ of $2$-dimensional $F$-manifolds. From the  comments which end Section 
\ref{gen-sect}
we  can (and will) assume, 
without loss of generality, that $(M, 0)$ is  either 
$I_{2}(m)$ $(m\geq 3$) or $\mathcal N_{2}.$

\section{Differential equations}\label{diff-sect}

We now prove various results on differential equations which will be useful in the next sections.
Along this section $t\in (\mathbb{C},0)$ is the standard coordinate.

\begin{lem}\label{elementary}
Consider the differential equation
\begin{equation}\label{diff-eqn}
\frac{d}{dt} ( ah) + a\frac{dh}{dt} = c,
\end{equation}
where $a, c\in \mathbb{C} \{ t\} $ are given 
and the function $h= h(t)$ is unknown.

i)  If $a(0) \neq 0$,  then  there is a unique formal solution $h$ with given $h(0)\in \mathbb{C}$ and this solution is holomorphic. \

ii) If $t=0$ is a zero of order one for $a$, then there is a unique formal solution of (\ref{diff-eqn}) and this solution is holomorphic.\

iii)  If $t=0$ is a zero of order  $o\geq 2$ for $a$, then (\ref{diff-eqn})  has a formal solution if and only if $t=0$ is a zero of order at least $o-1$ for $c$. When it exists, the formal solution is unique and holomorphic.\

In all  cases, if $a$ and $c$ converge on $\Delta$ (an open  disc centred at  $0\in \mathbb{C}$), then also the formal solution converges on $\Delta .$ 
\end{lem}

\begin{proof} As the proof is elementary, we skip the details. Claim i) follows from the fundamental theorem of 
differential equations. For claims ii) and iii), one checks easily 
(by  taking power series and identifying coefficients) the part concerning the  existence of formal solutions. 
For the convergence, one  uses the general result that 
any formal solution $u(t)=\sum_{n\geq 0} u_{n} t^{n}$ of a differential equation of the form
\begin{equation}\label{ec-w}
t \dot{u}(t) + A(t) u(t) = b(t),
\end{equation}
where $A: \Delta  \rightarrow M_{n}(\mathbb{C})$ and 
$b: \Delta \rightarrow \mathbb{C}^{n}$
are holomorphic,  is convergent on $\Delta$.
This was proved e.g. in Theorem 5.3  of \cite{wasow} (see page 22), when $b=0.$ The case $b\neq 0$
can be reduced to the case $b=0$ in the standard way:  if $u = (u_{1}, \cdots  , u_{n})^{t}$ is a solution of (\ref{ec-w}) with $b\neq 0$, 
one defines $v:= (u_{1}, \cdots , u_{n}, 1)^{t}$  and sees that $v$ satisfies a differential
equation (in dimension $n+1$) of the same  type (\ref{ec-w}) but with $b=0$.
One easily shows that in claims ii) and iii)  
equation (\ref{diff-eqn}) reduces to an equation of  the form (\ref{ec-w})   (with $A$ and $b$  scalar functions).
We obtain  that the formal solution of 
(\ref{diff-eqn}), in these cases, converges on $\Delta$ if 
$a$ and $c$ do.
\end{proof}

For a given function $f: (\mathbb{C} , 0) \rightarrow (\mathbb{C} , 0)$ and $n\in \mathbb{Z}_{\geq 1}$ we denote by $f^{n}$ the function $f^{n}(t):= 
f(t) \cdot ...\cdot f(t)$ (multiplication $n$-times; not to be confused with the iterated composition 
$f\circ ... \circ f$).

\begin{lem}\label{adaug-ajutatoare} Let $f\in \mathbb{C} \{ t\}$ be non-trivial and $r:= \mathrm{ord}_{0}(f).$ Then there is $\lambda\in \mathbb{C} \{ t\}$, with
$\lambda (0) =0$ and $\dot{\lambda}(0) \neq 0$, such that $(\dot{\lambda})^{2} \lambda^{r} =f.$ Moreover, any two such functions  $\lambda$ and
$\tilde{\lambda}$ 
are related by $\tilde{\lambda} (t) = \lambda_{0}\lambda (t)$, where $\lambda_{0}\in \mathbb{C}$, $\lambda_{0}^{r+2} =1.$ 
\end{lem}

\begin{proof}  As $r= \mathrm{ord}_{0}(f)$, we  can write $f (t) = t^{r} g(t)$ with $g\in \mathbb{C} \{ t\}$ a unit. 
Similarly, the function $\lambda$ we are looking for is of the form $\lambda (t)  = t x(t)$, with
$x\in \mathbb{C} \{ t\}$ a unit.   
We are looking for $x$ which satisfies the differential equation
\begin{equation}\label{prev1}
( x + t \dot{x})^{2} x^{r} = g.
\end{equation}
As $g(0) \neq 0$, there is $k\in \mathbb{C} \{ t\}$ 
a unit,
such that $g = k^{2}$. 
Similarly, as $x(0) \neq 0$ we can write $x= z^{2}$, for $z\in \mathbb{C} \{ z\} .$ 
Equation (\ref{prev1})  is  satisfied if 
$(x+ t \dot{x})z^{r} = k$ or
\begin{equation}\label{prev2}
2t (z^{r+2})^{\prime} + (r+2) z^{r+2} = (r+2) k. 
\end{equation}
 The differential equation in the unknown function $y$ 
$$
2t \dot{y} + (r+2) y = (r+2) k
$$
has a unique formal solution.   From  Lemma \ref{elementary}, this solution is holomorphic. 
As $k(0)\neq 0$, we obtain  $y(0) \neq 0$. Let  $z\in \mathbb{C}\{ t\}$ such that $z^{r+2} =y.$
The function  $z$ satisfies (\ref{prev2}) and  $\lambda (t) := tz(t)^{2} $ 
satisfies $(\dot{\lambda})^{2} \lambda^{r} =f$, as needed.
The first statement is proved. 
The second statement follows by taking into account the freedom in the choice of $z$ and $k$ 
in the above argument.
\end{proof}

\begin{lem}\label{de-adaugat} 
Consider the system of two differential equations
\begin{align}
&\dot{ g}_1 + g_2 h_1+h_2=0,\label{3.1}\\
& (r+t\frac{d}{dt})(g_2)+ g_1  h_3+h_4=0,\label{3.2}
\end{align}
where $r\in \mathbb{R}_{>0}$, 
$h_1,h_2,h_3,h_4\in\mathbb{C}\{t\}$ are given,
and the functions $g_1=g_1(t)$ and $g_2=g_2(t)$ are unknown.

For any choice of $g_1(0)\in\mathbb{C}$, there
is a unique formal solution $(g_{1}, g_{2})$ 
of (\ref{3.3}), (\ref{3.4}),  and this is holomorphic.
If $h_1,h_2,h_3$ and $h_4$ converge on $\Delta$
(an open disc centred at $0\in\mathbb{C}$), 
then also $g_1$ and $g_2$ converge on $\Delta$.
\end{lem}

\begin{proof} 
Writing $g_i=\sum_{k\geq 0}g_{i}^{(k)}t^k$ and
$h_j=\sum_{k\geq 0}h_{j}^{(k)}t^k$, 
the two differential equations are equivalent to the 
following equations,
\begin{align}
&n  g_1^{(n)}+ \sum_{k=0}^{n-1}g_2^{(k)}h_1^{(n-1-k)}
+h_2^{(n-1)}=0, \quad \forall n\in\mathbb{Z}_{\geq 1},
\label{3.3}\\
& (r+n)g_2^{(n)}+ \sum_{k=0}^n g_1^{(k)}h_3^{(n-k)}
+h_4^{(n)}=0, \quad \forall n\in\mathbb{Z}_{\geq 0}.
\label{3.4}
\end{align}
Let  $g_{1}^{(0)}\in\mathbb{C}$ be given.  Then  equations  (\ref{3.3}), (\ref{3.4}) determine
inductively all coefficients $g_1^{(n)}$,  for $n\geq 1$, 
and $g_2^{(n)}$,  for $n\geq 0$. We obtain a unique
formal solution $(g_1, g_2)$. From the proof of Lemma \ref{elementary}, $g_{1}$ and $g_{2}$  are holomorphic
(multiplying (\ref{3.1})  by $t$ we notice that (\ref{3.1}), (\ref{3.2}) is of the form 
(\ref{ec-w})). 
\end{proof}

\section{$(T)$-structures  over $I_{2}(m)$}\label{formal-i2}

In  this section we find  formal normal forms for $(T)$-structures 
over the germ $I_{2}(m)$ ($m\geq 3$).   We need to introduce  notation.

\begin{notation} {\rm  Along this section
$(t_{1}, t_{2})$ denote the standard coordinates on $\mathbb{C}^{2}$. 
We shall use the following matrices
\begin{equation}\label{6.3}
C_1:= \mathrm{Id }_2,\ 
C_2:=\left(\begin{tabular}{cc} 
$0$ & $t_{2}^{m-2}$\\
$1$ & $0$
\end{tabular}\right),\  
D:= \left( \begin{tabular}{cc}
$1$ & $0$\\
$0$ & $-1$
\end{tabular}\right) , \ 
E := \left( \begin{tabular}{cc}
$0$ & $1$\\
$0$ & $0$\end{tabular}\right) , 
\end{equation}
and the relations between them:
\begin{align}
\nonumber& (C_2)^{2}=t_2^{m-2}C_1,\  D^{2}  =C_1,\  E^{2}=0,\\
\nonumber & C_2  D =C_2-2t_2^{m-2}E = - D  C_2,\\
\nonumber &   C_2  E=\frac{1}{2} (C_{1}-D),\  E  C_2 = \frac{1}{2} (C_{1} + D),\\
\label{r4}& D  E=E =-E  D .
\end{align}
Remark that 
\begin{equation}\label{commutators}
[C_{2}, D]= 2( C_{2} - 2 t_{2}^{m-2} E),\  [C_{2}, E]= -D,\   [D, E] = 2E.
\end{equation}
The matrices  $C_{1}$, $C_{2}$, $D$ and $E$ form an $\mathcal O_{(\mathbb{C}^{2},0)}$-basis of 
$M(2\times 2,\mathcal O_{(\mathbb{C}^{2},0)})$.}
\end{notation}

Theorem \ref{t4.1} i) produces a {\it non-unique}  normal
form, with respect to gauge isomorphisms,  for any $(T)$-structure over  $I_2(m)$ $(m\geq 3)$.
Theorem \ref{t4.1} iii)  produces  a {\it unique}  normal form, 
with respect to {\it formal} gauge isomorphisms, 
for any formal or holomorphic $(T)$-structure over such a germ. 

\begin{thm}\label{t4.1}
i) Over the germ $I_2(m)$ $(m\geq 3)$, any $(T)$-structure
is gauge isomorphic to a $(T)$-structure of the form
\begin{equation}\label{4.1}
A_1=C_1,\ A_2=C_2+zf E,
\end{equation}
where $f\in\mathbb{C}\{t,z\}$ is holomorphic.

ii) Any formal $(T)$-structure over  $I_{2}(m)$  is
formally gauge isomorphic to a formal $(T)$-structure of 
the form \eqref{4.1} where $f\in\mathbb{C}\{t,z]]$. \ 

iii)  Any holomorphic or formal $(T)$-structure over $I_{2}(m)$  is formally gauge isomorphic to a formal 
$(T)$-structure of the form \eqref{4.1} where
$f=0$ if $m=3$ and
$f\in \mathbb{C}[[z]] [t_{2}]_{\leq m-4}$ if $m\geq 4$.
With respect to formal gauge isomorphisms, the function $f$
is unique.
\end{thm}

\begin{proof} 
To prove claim i), we  start with an arbitrary $(T)$-structure $(H, \nabla ) $ 
over $I_2(m)$ with $m\geq 3.$  
We choose a trivialization 
$\underline{s}=(s_1,s_2)$ of $H$ 
such that the connection
form $\Omega$ of $\nabla$ is given by $\Omega = \frac{1}{z} ( A_{1} dt_{1}+ A_{2} dt_{2})$, where 
\begin{equation}\label{6.11}
A_1(0)= C_1,\quad A_2(0)= C_2.
\end{equation}
We will reduce $\nabla$  to the required (non-unique) 
normal form in three steps.
 
\medskip
The {\bf first step} of the normalization is the reduction of
$A_1$ to $C_1$ and of $A_2$ to a new matrix 
$\widetilde{A}_2$ with $\widetilde{A}_2(0)=C_2$
and $\partial_1 \widetilde{ A}_2=0$. Consider the system 
\begin{equation}\label{6.12}
\partial_1 T=- (\sum_{k\geq 1} A_1(k) z^{k-1} )T,\  T(z, 0, t_{2})=C_{1}. 
\end{equation}
It  has a unique holomorphic  solution $T$. 
We write $T = \sum_{k\geq 0} T(k) z^{k}$ with $T(k)$ independent of $z$. 
We claim that 
\begin{equation}\label{6.14}
T(0)\in\mathcal O_{(\mathbb{C}^{2},0)}\cdot C_1+\mathcal O_{(\mathbb{C}^{2},0)}\cdot C_2.
\end{equation}
To prove this claim, we remark that 
 (\ref{6.12})  for $z=0$ gives 
\begin{equation}\label{6.15}
\partial_1 T(0)= -A_1(1)T(0),\ T(0)(0,t_2) =C_{1}.
\end{equation}
On the other hand,  relation (\ref{2.12})  for $k=1$ 
together with (\ref{6.11})  gives
\begin{equation}
0=\partial_1 A_2(0)-\partial_2 A_1(0)+[A_1(0),A_2(1)]+[A_1(1),A_2(0)]=  [A_1(1),C_2], 
\end{equation}
which implies  $A_1(1)=a_1C_1+a_2C_2$ for 
$a_1,a_2\in\mathcal O_{(\mathbb{C}^{2},0)}.$ 
The differential equation (\ref{6.15})  
with $A_{1}(1)$ of this form 
and ansatz 
$T(0)=\tau_{01}C_1+\tau_{02}C_2$,  with $\tau_{01},\tau_{02}\in \mathcal O_{(\mathbb{C}^{2},0 )}$, 
and $\tau_{01}(0, t_{2})=1$, $\tau_{02} ( 0, t_{2}) =0$,  
has a unique solution. We obtain that  $T$ satisfies (\ref{6.14}), as required. 
We now change the trivialization $\underline{s}$ by means of  $T$.   
In the new trivialization, $\nabla$
is given  by  matrices $\widetilde{A}_{1}$ and $\widetilde{A}_{2}.$     From  
(\ref{2.15}) for $i=1$, together with $A_{1}(0) = C_{1}$ and (\ref{6.12}), 
we obtain:
\begin{equation}
0= z\partial_1 T+A_1T-T \widetilde{A}_1 = C_1T-T \widetilde{A}_1 = T (C_1-\widetilde{ A}_1)
\end{equation}
which implies  $\widetilde{A}_1=C_1.$  From (\ref{2.17})  for $k=0$ and $i=2$,
\begin{equation}
0= A_2(0) T(0)-T(0) \widetilde{A}_2(0) = T(0) (C_2-\widetilde{ A}_2(0)),
\end{equation}
where we used (\ref{6.14}) and $A_{2}(0) = C_{2}.$  We obtain
$\widetilde{ A}_2(0)=C_2$. Finally, from  (\ref{2.17-w}),
\begin{equation}
0= z\partial_1 \widetilde{A}_2-z\partial_2 \widetilde{A}_1+[\widetilde{A}_1,\widetilde{A}_2]=z\partial_1\widetilde{ A}_2,
\end{equation}
from which we deduce 
$\partial_{1}\widetilde{A}_{2} =0.$ The first step is completed.\

Owing to the first step, from now on we assume  that  $A_1=C_1$, $A_2(0)=C_2$ and $\partial_1 A_2=0$.

\medskip
The {\bf second step} does not change $A_1=C_1$ and 
erases the term $C_1$ in $A_2$. Suppose 
that 
\begin{equation}\label{6.16}
A_2= C_2+z(a_1C_1+a_2C_2+a_3D +a_4E)
\end{equation}
with $a_1,a_2,a_3,a_4\in \mathbb{C}\{z,t_2\}.$ 
Let  $\tau_1\in
\mathbb{C} \{z,t_2\}$ be the unique solution of 
\begin{equation}\label{6.17}
\partial_2 \tau_1=-a_1\tau_1,\ \tau_1(z,0)=1
\end{equation}
and  $T: =\tau_1 C_1$. Relation  (\ref{2.17-w})  for $i=2$ 
gives
$$ 0=z\partial_2 T+A_2 T-T\widetilde{ A}_2  \\
= (C_2+z(a_2C_2+a_3D+a_4E))-\widetilde{ A}_2)T. 
$$
Thus
\begin{equation}\label{6.18}
 \widetilde{ A}_2 = C_2+z(a_2C_2+a_3D+a_4E),
\end{equation}
as needed. Remark that the coefficients of $C_{2}$, $D$ and $E$ in the expressions 
(\ref{6.16}) and (\ref{6.18}) 
of $A_{2}$ and $\tilde{A}_{2}$ are the same.

\medskip
The {\bf third step} of the reduction does not change $A_1=C_1$
and brings $A_2$ to the form $C_2+z  f  E$ with
$f\in\mathbb{C} \{z,t_2\}$. Suppose 
\begin{equation}\label{6.19}
A_2=C_2+z(a_2C_2+a_3D+a_4E)
\end{equation}
with $a_2,a_3,a_4\in \mathbb{C}\{z,t_2\}.$ 
We are searching for $T$ and  $\widetilde{A}_2$ of the form 
\begin{align}
\label{6.20} &T= C_1 + z(\tau_3D+\tau_4E)\\
\label{6.21} &  \widetilde{ A}_2 = C_2+z(\tilde{ a}_1C_1+\tilde{a}_4E)
\end{align}
where  $\tau_3,\tau_4, \tilde{ a}_1,\tilde{a}_4\in\mathbb{C}\{ t_2, z \}$, which,  together with $A_2$, satisfy 
(\ref{2.17-w})  for $i=2$:
\begin{align}
\label{s3-add}&0= z\partial_2 T+A_2  T-T  \widetilde{A}_2 \\
\nonumber&= z^2\partial_2\tau_3D +z^2\partial_2\tau_4E + z[C_2,\tau_3D+\tau_4E]\\
\nonumber&+ (C_2+z(a_2C_2+a_3D+a_4E))-(C_2+z(\tilde{ a}_1C_1+\tilde{a}_4E))\\
\nonumber&+ z^2(a_2C_2+a_3D+a_4E)(\tau_3D+\tau_4E) - z^2(\tau_3D+\tau_4E)(\tilde{a}_1C_1+\tilde{a}_4E )\\
\nonumber&= z^2 \partial_2\tau_3D+z^2\partial_2\tau_4E +z(\tau_3(2C_2-4t_2^{m-2}E)- \tau_4 D)\\
\nonumber&+ z(a_2C_2+a_3D+a_4E)-z(\tilde{a}_1C_1+\tilde{a}_4E)\\
\nonumber&+ z^2(a_2\tau_3(C_2-2t_2^{m-2}E)+a_3\tau_3C_1- a_4\tau_3 E\\
\nonumber& + a_2\tau_4\frac{1}{2}(C_1-D)+a_3\tau_4E) - z^2(\tilde{a}_1\tau_3D+\tilde{a}_1\tau_4E_4+\tilde{a}_4\tau_3E).
\end{align}
Ordering the terms and dividing once by $z$, we obtain
\begin{align}
\label{6.22} & 0= C_1 (-\tilde{a}_1+z(a_3\tau_3+\frac{1}{2}a_2\tau_4)) + C_2 (2\tau_3+a_2+za_2\tau_3)\\
\nonumber&  +  D  (-\tau_4+a_3+z(\partial_2\tau_3-\frac{1}{2}a_2\tau_4-\tilde{a}_1\tau_3))\\
\nonumber&+ E \left( a_4-\tilde{a}_4-4t_2^{m-2}\tau_3  +z(\partial_2\tau_4-2t_2^{m-2}a_2\tau_3
-(a_4+\tilde{ a}_4)\tau_3+a_3\tau_4-\tilde{a}_1\tau_4)\right).
\end{align}
The coefficient of $C_2$ determines $\tau_3$ uniquely 
($2+za_2$ is a unit in $\mathbb{C}\{z,t_2\}$). The coefficient of $C_{1}$ determines
$\tilde{a}_{1}$ in terms of $\tau_{4}.$   The coefficient of $D$ determines then $\tau_{4}.$ Finally,
the coefficient of $E$ determines $\tilde{a}_{4}.$ 
We proved that $A_{2}$ can be brought to the form  (\ref{6.21}). Applying the second step to $A_{1} = C_{1}$ and 
$\tilde{A}_{2}$ given by (\ref{6.21}), we bring  (without changing $A_{1} = C_{1}$) $A_{2}$  to the form
$C_{2} + z f E$ (with $f=\tilde{a}_{4}$),  as needed.
This completes the proof of claim  i).\

\bigskip 
The proof of claim  ii) is analogous, with series in $\mathbb{C}\{t_2,z]]$
instead of functions in $\mathbb{C}\{t_2,z\}$.\

\bigskip
Now we prove claim iii). For this let two arbitrary formal normal forms $A_1,A_2$ and 
$\tilde{A_1},\tilde{A_2}$, be given by 
\begin{equation}\label{4.1-2}
A_1=\tilde{A_1}=C_1,\quad 
A_2=C_2+zfE,\quad \tilde{A_2}=C_2+z\tilde{f}E
\end{equation}
where $f,\tilde{f}\in\mathbb{C}\{t_2,z]]$. We study when 
they are formally gauge  isomorphic. This happens if there is 
a matrix-valued power series 
\begin{equation}\label{4.1-3}
T= \tau_1C_1+\tau_2C_2+\tilde{\tau_3}D+\tilde{\tau_4}E
\end{equation}
with $\tau_1,\tau_2,\tilde{\tau_3},\tilde{\tau_4}\in\mathbb{C}\{t,z]]$,
such that $T(0)(0)$ is invertible 
and 
\begin{equation}\label{4.2}
z\partial_j T + A_j T-T\tilde{A_j}=0,\quad j\in\{1,2\}.
\end{equation}
Relation \eqref{4.2} for $j=1$ gives  $\partial_1T=0$, or 
$\tau_1,\tau_2,\tilde{\tau_3},\tilde{\tau_4}
\in \mathbb{C}\{t_2,z]]$.
We write $\tau_i=\sum_{n\geq 0}\tau_i(n)z^n$ with 
$\tau_i(n)\in\mathbb{C}\{t_2\}$ 
($i\in \{ 1,2\}$) 
and similarly
for $\tilde{\tau_i}$
($i\in \{ 3,4\}$). 
Relation \eqref{4.2} for $j=2$ gives 
\begin{align}
\nonumber&0= z\partial_2 T+A_2  T-T  \widetilde{A}_2 \\
\nonumber&= z\left(  (\partial_2\tau_1 ) C_1 
+ (\partial_2\tau_2 ) C_2+(m-2)  t_2^{m-3}\tau_{2}E
+(\partial_2\tilde{\tau_3})D +(\partial_2\tilde{\tau_4}) E\right)\\
\nonumber&+ [C_2,T] + z (fET-\tilde{f}TE)\\
\nonumber& =  z\left(  (\partial_2\tau_1 ) C_1 
+ (\partial_2\tau_2 ) C_2
+(\partial_2\tilde{\tau_3})D+((m-2)  t_2^{m-3}\tau_{2}+\partial_2\tilde{\tau_4}) E\right)\\
\nonumber & + 2 \tilde{\tau}_{3} ( C_{2} - 2 t_{2}^{m-2} E) -\tilde{\tau}_{4} D\\
\label{intermediar} & + z\left( \frac{f\tau_{2}}{2} (C_{1}+D) + f (\tau_{1} -\tilde{\tau}_{3}) E -\frac{\tilde{f} \tau_{2}}{2} (C_{1} - D)
-\tilde{f} (\tau_{1} +\tilde{\tau}_{3})E\right) ,
\end{align}
where we used relations  (\ref{r4}) and (\ref{commutators}).
The above relation implies that $\tilde{\tau}_{3}(0)=\tilde{\tau}_{4}(0)=0$. Therefore, we can write 
$\tilde{\tau}_{i} (z) =z \tau_{i}(z)$
where $\tau_{i}\in \mathbb{C} \{ t_{2}, z]]$  ($i=3,4$).  
The coefficients of $C_2$ and $D$ in (\ref{intermediar})  determine
$\tau_3$ and $\tau_4$ in terms of $\tau_2$:
\begin{align}
\tau_3 &= -\frac{1}{2}\partial_2\tau_2,\label{4.3}\\
\tau_4 &= \frac{\tau_{2} }{2} (f+\tilde{f})
-\frac{z}{2}\partial_2^2\tau_2\label{4.4}
\end{align}
The coefficients of $C_1$ and $E$ in (\ref{intermediar}) give for $\tau_1$ and 
$\tau_2$ the system of differential equations
\begin{align}\label{4.5}
&\partial_2\tau_1 + \frac{\tau_{2}}{2} (f-\tilde{f})=0,\\
& t_2^{m-3}((m-2)+2t_2\partial_2)(\tau_2)
+\tau_1(f-\tilde{f})\label{4.6}\\
\nonumber &+\frac{z}{2}\Bigl(\partial_2(\tau_2(f+\tilde{f}))
-z\partial_2^3\tau_2
+(f+\tilde{f})\partial_2\tau_2\Bigr) =0.
\end{align}

Now suppose that $f\in\mathbb{C}\{t_2,z]]$ is given.
We claim  that there exist solutions
$(\tau_1,\tau_2,\tilde{f})$ of \eqref{4.5},  \eqref{4.6}
with $\tau_1,\tau_2\in\mathbb{C}\{t_2,z]]$,  
$\tau_1|_{t_2=0}\in\mathbb{C}[[z]]^*$ arbitrary and
$\tilde{f}\in \mathbb{C}[[z]][t_2]_{\leq m-4}$ if $m\geq 4$
respectively $\tilde{f}=0$ if $m=3$.
We only prove the statement for $m\geq 4$ (the statement for $m=3$ can be proved similarly). 
We write $f=\sum_{n\geq 0}f(n)z^n$ 
with $f(n)\in\mathbb{C}\{t_2\}$,
and similarly 
$\tau_{i} =\sum_{n\geq 0} \tau_{i}(n) z^{n}$ ($i\in \{ 1,2\}$),   $\tilde{f}=\sum_{n\geq 0}\tilde{f}(n)z^n$ 
where $\tau_{i}(n)$  and $\tilde{f}(n)$  depend only on $t_{2}.$ 
Let $\Delta$ be an open disc centered at $0\in\mathbb{C}$
with $f(n)\in\mathcal{O}_\Delta$ for all $n\geq 0$. 

Relations \eqref{4.5},  \eqref{4.6}
give for $\tau_1(0),\tau_2(0)$ and $\tilde{f}(0)$ the system of equations
\begin{align}
& \frac{d}{dt_{2}} \tau_1(0)+ \frac{1}{2}\tau_2(0)(f(0)-\tilde{f}(0))=0, 
\label{4.7}\\
&t_2^{m-3}((m-2)+2t_2\frac{d}{dt_{2}} )(\tau_2(0))
+\tau_1(0)(f(0)-\tilde{f}(0))=0.\label{4.8}
\end{align}
As  $\tau_1(0)(0)\in\mathbb{C}^*$, 
solvability
of \eqref{4.8} requires that $t_2^{m-3}$ divides $f(0)-\tilde{f}(0)$.
Since $\tilde{f}(0)\in\mathbb{C}[t_2]_{\leq m-4}$, we obtain  $\tilde{f}(0)=[f(0)]_{\leq m-4}$.
Therefore,  $f(0)-\tilde{f}(0)=[f(0)]_{\geq m-3}$. 
After dividing (\ref{4.8}) by $t_{2}^{m-3}$, the system (\ref{4.7}), (\ref{4.8}) takes the form 
(\ref{3.1}), (\ref{3.2}). 
Using 
Lemma \ref{de-adaugat} we obtain,  for each
value $\tau_1(0)(0)\in\mathbb{C}^*$, a  unique formal solution
$(\tau_1(0),\tau_2(0) )$. This solution  is holomorphic on  $\Delta$.

Consider now $n\geq 1$ and assume that  
$\tau_1(k),\tau_2(k)\in\mathcal{O}_\Delta$ 
and $\tilde{f}(k)\in \mathbb{C}[t_2]_{\leq m-4}$
are known, for any  $k\leq n-1$,
such that \eqref{4.5}, \eqref{4.6} hold up to order $n-1.$ 
The coefficients of $z^n$ in \eqref{4.5}, \eqref{4.6}
give for $\tau_1(n),\tau_2(n)$ and $\tilde{f}(n)$ 
the system of equations
\begin{align}
\label{4.9}& \frac{d}{dt_{2}} \tau_1(n)+ \frac{1}{2}\tau_2(n)[f(0)]_{\geq m-3}
+ \frac{1}{2}\tau_2(0)(f(n)-\tilde{f}(n))+  h_1(n)=0\\
\nonumber&t_2^{m-3}((m-2)+2t_2\frac{d}{dt_{2}} )(\tau_2(n))
+\tau_1(0)(f(n)-\tilde{f}(n))+\tau_1(n)[f(0)]_{\geq m-3}\\
\label{4.10}&+ h_2(n)=0,
\end{align}
where $h_1(n),h_2(n)\in\mathcal{O}_\Delta$ are known functions, which 
depend  on $\tau_{1}(k)$, $\tau_{2}(k)$ and $\tilde{f}(k)$
for $k\leq n-1$.  Solvability of \eqref{4.10} requires that 
$t_2^{m-3}$ divides $\tau_1(0)(f(n)-\tilde{f}(n))+h_2(n)$.
Let  $\tilde{f}(n)\in\mathbb{C}[t_2]_{\leq m-4}$ 
be the unique polynomial of degree at most $m-4$ with this property.
With this definition of $\tilde{f}(n)$, 
equations \eqref{4.9},  \eqref{4.10} reduce 
(after dividing (\ref{4.10}) by $t_{2}^{m-3}$),  
to the system of equations
\begin{align}
\label{4.11}&\frac{d}{dt_{2}}\tau_1(n)+ \frac{1}{2}\tau_2(n)[f(0)]_{\geq m-3}
+h_3(n)=0\\
\label{4.12}& ((m-2)+2t_2\frac{d}{dt_{2}})(\tau_2(n))
+\tau_1(n) [f(0)]_{\geq m-3} t_2^{-(m-3)} + h_4(n) =0,
\end{align}
where $h_3(n),h_4(n)\in\mathcal{O}_\Delta$ are known.
Lemma \ref{de-adaugat} applies to the system
\eqref{4.11}, \eqref{4.12} 
and gives, for each
value $\tau_1(n)(0)\in\mathbb{C}$,  a unique solution
$(\tau_1(n),\tau_2(n))$ which is holomorphic on $\Delta$. 
The existence of a solution $(\tau_{1}, \tau_{2}, \tilde{f})$  
for (\ref{4.9}), (\ref{4.10}) 
with $\tau_{i} \in \mathbb{C} \{ t_{2}, z]]$ 
and $\tilde{f}\in \mathbb{C} [[z]][t_{2}]_{\leq m-4}$ follows by induction.

It remains  to prove the uniqueness part of claim iii). 
Suppose that  
$(\tau_1,\tau_2,f,\tilde{f})$ satisfy  \eqref{4.5},  \eqref{4.6}
and  $\tau_1,\tau_2\in\mathbb{C}\{t_2,z]]$,  
$f,\tilde{f}\in\mathbb{C}[[z]][t_2]_{\leq m-4}$.
We need to prove that $f =\tilde{f}.$ 
Equation (\ref{4.8}) together with $\tau_{1}(0)(0) \neq 0$ and $f(0), \tilde{f}(0) \in \mathbb{C}[t_{2}]_{\leq m-4}$  
implies that  $f(0)=\tilde{f}(0)$ and $\tau_2(0)=0$. 
Then, equation (\ref{4.7}) implies that  
$\tau_1(0)\in\mathbb{C}^*$. Consider now $n\geq 1.$ Assume that 
$f(k)=\tilde{f}(k),\tau_1(k)\in\mathbb{C}$ and $\tau_2(k)=0$
for $k\leq n-1$. 
As $h_{1}(n)= \frac{1}{2} \sum_{k=1}^{n-1} \tau_{2}(n-1-k) ( f(k) -\tilde{f}(k))$  we obtain
that $h_1(n)=0.$
Similarly, 
$h_2(n)=0$. 
Equation (\ref{4.10}), together with $[f(0)]_{\geq m-3}=0$,  $h_{2}(n)=0$ and $\tau_{1}(0)(0) \in \mathbb{C}^{*}$, 
implies, as before,   that $f(n) =\tilde{f}(n)$ and $\tau_{2}(n) =0.$ 
Equation (\ref{4.9}) implies that $\tau_{1}(n)\in \mathbb{C}.$  
Inductively we obtain $f=\tilde{f},\tau_1\in\mathbb{C}[[z]]^*$
and $\tau_2=0$. 
This finishes the proof of claim iii). 
\end{proof}

\begin{rem}{\rm i) The germs $I_{2}(m)$ ($m\geq 3$)  coincide with the germs at the origin of the orbit spaces  
$\mathbb{C}^{2}/W$ of various Coxeter groups  $W$, 
with their natural $F$-manifold structure (see \cite{Hbook}, page 19). 
In particular,   $W=A_{2}$ for $I_{2}(3)$,  $W=B_{2}=C_{2}$ for $I_{2}(4)$,  $W=H_{2}$ for $I_{2}(5)$ and $W=G_{2}$ for $I_{2}(6)$ (see \cite{bou, co} for the definition and classification of Coxeter groups
and  Lecture 4 of \cite{Dub},  reference \cite{Dub-cox}, 
Theorem 14 of \cite{gi2},  or Theorem 5.18 of \cite{Hbook} 
for the  natural $F$-manifold structure on their orbit spaces).   An immediate consequence of Theorem \ref{t4.1} iii)  is that any two formal $(T)$-structures over the germ at the origin of 
$\mathbb{C}^{2}/ A_{2}$ are formally isomorphic.\

ii) The  multiplication $\circ$ of $I_{2}(m)$ 
underlies a Frobenius manifold structure. This follows from the general fact that  the $F$-manifold multiplication of  the orbit space 
of a Coxeter group can be extended to a Frobenius manifold structure 
(see  Lecture 4 of \cite{Dub} or reference 
\cite{Dub-cox}).
Therefore, over $I_{2}(m)$   
lies a standard $(T)$-structure 
\begin{equation}\label{T-standard}
\nabla_{X}Y := D^{\tilde{g}}_{X}Y + \frac{1}{z} X\circ Y,\  \forall X, Y\in {\mathcal T}_{M},
\end{equation}
where $D^{\tilde{g}}$ is the Levi-Civita connection of the Frobenius  metric $\tilde{g}= dt_{1} \otimes dt_{2} + dt_{2} \otimes dt_{1}$. 
 This standard $(T)$-structure  coincides with the  normal form \eqref{4.1} with $f=0$.   Let us consider now a normal form \eqref{4.1}  with $f\in \mathbb{C} [ t_{2}]_{\leq m-4}$ (i.e. $f$
independent of $z$). It is mapped, by means of the gauge  isomorphism
$$
T=\left( \begin{tabular}{cc}
$1$ & $\beta$\\
$0$ & $1$
\end{tabular}
\right)
$$
with $\beta = \beta (t_{2}) \in \mathbb{C} [t_{2}]_{\leq m-3}$ such that $\dot{\beta} = -f$, 
to the $(T)$-structure with 
\begin{equation}\label{T-cox}
\tilde{A}_{1} = C_{1},\  
\tilde{A}_{2} =T^{-1}A_2T+T^{-1}z\partial_2 T = 
 \left( \begin{tabular}{cc}
$-\beta $ & $ t_{2}^{m-2}-\beta^2$\\
$1$ & $\beta $
\end{tabular}
\right) . 
\end{equation}
Remark that both $\tilde{A}_{1}$ and $\tilde{A}_{2}$ in (\ref{T-cox}) are independent on $z$.
This is an example of a general result, namely that any $(T)$-structure is 
locally (holomorphically) isomorphic to a $(T)$-structure with connection form
$\frac{1}{z} \sum_{i} \tilde{A}_{i} dt_{i}$  where  $\tilde{A}_{i} = \tilde{A}_{i}(0)$ are independent on $z$.
We shall prove this result in a forthcoming paper.}
\end{rem}

\section{$(T)$-structures over $\mathcal N_{2}$}\label{formal-n2}

In this section we find formal normal forms for  $(T)$-structures over $\mathcal N_{2}$.   
In a first stage we will  find them  up to formal gauge isomorphisms. They are described in 
Theorem \ref{t5.1}, whose proof  relies  on the calculations from the proof of Theorem \ref{t4.1}. 
In a second stage  we will exploit the additional freedom
from holomorphic isomorphisms which lift non-trivial automorphisms of the base.  Theorem \ref{t5.2} states the results. 
Finally, Theorem \ref{t5.3} combines  Theorems \ref{t5.1}
and \ref{t5.2} and gives round formal normal forms for $(T)$-structures
over $\mathcal{N}_2$  up to the entire group of formal isomorphisms.

We will use the same matrices $C_{1}$, $D$ and $E$, as in the previous section.
The definition of the matrix  $C_{2}$ is almost the same as before, with the only difference that the  $(1,2)$ entry
$t_2^{m-2}$ is replaced by $0$. Thus
\begin{equation}\label{5.1}
C_1:=\textup{Id}_2,\ 
C_2:=\begin{pmatrix}0&0\\1&0\end{pmatrix},\ 
D:=\begin{pmatrix}1&0\\0&-1\end{pmatrix},\ 
E:=\begin{pmatrix}0&1\\0&0\end{pmatrix}.
\end{equation}
Relations \eqref{r4} and \eqref{commutators} still hold, with  $t_2^{m-2}$  replaced by $0$.

\begin{thm}\label{t5.1}
i) Over  $\mathcal{N}_2$, any $(T)$-structure is gauge
isomorphic to a $(T)$-structure of the form \eqref{4.1}
where $f\in\mathbb{C}\{t_2,z\}$ is holomorphic.\

ii) Any formal $(T)$-structure over  $\mathcal{N}_2$ 
is formally gauge isomorphic to a formal $(T)$-structure 
of the form \eqref{4.1} where $f\in\mathbb{C}\{t_2,z]]$.\

iii) For a holomorphic or formal $(T)$-structure 
of the form (\ref{4.1}) over $\mathcal N_{2}$, with $f\in \mathbb{C} \{ t_{2}, z\}$ respectively
$f\in \mathbb{C} \{ t_{2}, z]]$, the function $f(0)\in\mathbb{C}\{t_2\}$
is a formal gauge invariant of it. 
If $f(0)=0$, then also the function $f(1)\in\mathbb{C}\{t_2\}$
is a formal gauge invariant of it.\

iv) A holomorphic $(T)$-structure of the form (\ref{4.1}) 
over $\mathcal N_{2}$, with $f(0)=0$, 
is gauge isomorphic to a unique $(T)$-structure of the form
\begin{equation}
A_1=C_1,\ A_2=C_2+z^2f(1)E.
\end{equation}

v) A holomorphic or formal $(T)$-structure of the form (\ref{4.1})  
with $f(0)\neq 0$ and $\textup{ord}_0f(0)=r\in\mathbb{Z}_{\geq 0}$
is formally gauge isomorphic to the $(T)$-structure of the form
\begin{equation}\label{5.1-2} 
A_1 =C_1,\ A_2=C_2+zf(0)E
\end{equation}
if $r\in\{0, 1 \}$,  or to a formal $(T)$-structure of the form  
\begin{equation}\label{n-f-u}
A_1=C_1,\ A_2=C_2+z\tilde{f}E
\end{equation}
where  $\tilde{f}(0)=f(0)$ and  $\tilde{f}(k)\in \mathbb{C}[t_2]_{\leq r-2}$, 
for any $k\geq 1$, if  $r\geq 2.$
These  normal forms are  formal gauge invariants and are unique. 
\end{thm}

\begin{proof} Claims  i) and ii) follow with the same argument as steps 1-3 from the proof
of Theorem \ref{t4.1}  (the only difference lies in relations 
(\ref{s3-add})  and (\ref{6.22}), in which the terms  containing  
with $t_{2}^{m-3}$ are now replaced by $0$).\ 

To prove claim iii), we consider two formal normal forms $A_1,A_2$ and $\tilde{A_1},
\tilde{A_2}$ as in  \eqref{4.1-2} and a matrix
$T=\tau_1C_1+\tau_2T_2+\tilde{\tau_3}D+\tilde{\tau_4}E$
as in \eqref{4.1-3} such that relation  \eqref{4.2} is satisfied
(i.e. the formal normal forms are formally gauge isomorphic).
We find again that 
$T$ is independent on $t_{1}$, 
$\tilde{\tau_3}(0)=0,\tilde{\tau_4}(0)=0$
and thus write again 
$\tilde{\tau_3}(z)=z\tau_3(z)$ and $\tilde{\tau_4}(z)=z\tau_4(z)$
with $\tau_{3}$, $\tau_{4}\in \mathbb{C} \{ t_{2}, z]]$. 
The same calculations lead to the same equations
\eqref{4.3},  \eqref{4.4},  which determine
$\tau_3$ and $\tau_4$ in terms of $\tau_2,f,\tilde{f}$, 
and to the system of equations 
\begin{align}
& \partial_2\tau_1+ \frac{\tau_{2}}{2} (f-\tilde{f})=0,\label{5.2}\\
& \tau_1(f-\tilde{f})
+\frac{z}{2}\Bigl(\partial_2(\tau_2(f+\tilde{f}))
-z\partial_2^3\tau_2
+(f+\tilde{f})\partial_2\tau_2\Bigr) =0.\label{5.3}
\end{align}

As $\tau_1\in\mathbb{C}\{t_2,z]]$ is invertible, relation 
\eqref{5.3} implies that $f(0)=\tilde{f}(0)$. We obtain that 
$f(0)$ is a formal gauge invariant, as needed.  Suppose now that $f(0)=\tilde{f}(0)=0$. Identifying the coefficient of $z$ in 
\eqref{5.3} we obtain $\tau_{1}(0) (f(1) -\tilde{f}(1)) =0$. Since $\tau_{1}(0)\in \mathbb{C} \{ t_{2}\}$ is a unit, we deduce
that $f(1)=\tilde{f}(1)$, i.e. 
$f(1)$ is a formal gauge invariant. The proof of claim iii) is completed.\

We now prove claim iv). For this  we consider a $(T)$-structure of the form \eqref{4.1} 
with $f\in\mathbb{C}\{t_2,z\}$ and $f(0)=0$.
Define $\tilde{f}:=zf(1)$. We claim that the system 
of differential equations \eqref{5.2},  \eqref{5.3}
has a holomorphic solution $(\tau_1,\tau_2)\in 
\mathbb{C}\{t_2,z\}^2$ with $\tau_1|_{t_2=0}\in\mathbb{C}\{z\}^*$
arbitrary. 
Write $f-\tilde{f}=z^2g$ with $g\in\mathbb{C}\{t_2,z\}$.
Then \eqref{5.2}, \eqref{5.3} become the system
\begin{align}\label{5.4}
&\partial_2\tau_1+ \frac{1}{2}z^2\tau_2g =0,\\
&\partial_2^3\tau_2 -2\tau_1g
-\partial_2(\tau_2(2f(1)+zg))
-(2f(1)+zg)\partial_2\tau_2 =0,\label{5.5}
\end{align}
with leading parts $\partial_2\tau_1$ and $\partial_2^3\tau_2$.
It can be rewritten as a system of linear differential equations
in $t_2$ with holomorphic parameter $z$,
and it  has a holomorphic solution $(\tau_1,\tau_2)$ 
with $\tau_1|_{t_2=0}\in\mathbb{C}\{z\}^*$ arbitrary. 
Claim iv) is proved (the uniqness  follows from claim iii)).\ 

To prove claim v), we  consider a holomorphic or formal 
$(T)$-structure of the form \eqref{4.1} over $\mathcal{N}_2$
with $f\in\mathbb{C}\{t_2,z]]$ and $\textup{ord}_0f(0)=r
\in \mathbb{Z}_{\geq 0}$. 
We need to show that there exists a solution 
$(\tau_1,\tau_2,\tilde{f})$ of \eqref{5.2}, \eqref{5.3}
with $\tau_1,\tau_2\in\mathbb{C}\{t_2,z]]$, 
$\tau_1|_{t_2=0}\in\mathbb{C}[[z]]^*$ arbitrary and
$\tilde{f}=f(0)$ if $r\in\{0, 1\}$ and
$\tilde{f}(0)=f(0)$, 
$\tilde{f}(k)\in\mathbb{C}[t_2]_{\leq r-2}$ for $k\geq 1$,  
if $r\geq 2$. 
Let $\Delta$ be an open disc centered at $0\in\mathbb{C}$
with $f(k)\in\mathcal{O}_\Delta$ for all $k\geq 0$. 

Equations \eqref{5.2}, \eqref{5.3} give for
$\tau_1(0),\tau_2(0)$ and $\tilde{f}(1)$ the equations
\begin{equation}\label{5.6}
\frac{d}{dt_{2}} \tau_1(0)= 0,\quad 
\tau_1(0)(f(1)-\tilde{f}(1))
+\frac{d}{dt_{2}}(\tau_2(0)f(0))+f(0)\frac{d}{dt_{2}}\tau_2(0)=0.
\end{equation}
Choose  $\tau_1(0)\in\mathbb{C}^*$ arbitrary.
Then, for any $\tilde{f}(1)$ fixed, the second relation  (\ref{5.6}) can be considered as a differential equation of the form
(\ref{diff-eqn})
in the unknown function $\tau_{2}(0).$ 
When  $r\in\{0, 1\}$, we define $\tilde{f}(1):=0$.
When  $r\geq 2$, we define  $\tilde{f}(1)\in\mathbb{C}[t_{2}]_{\leq r-2}$
to be the unique polynomial of degree at most
$r-2$  such that  $t_2^{r-1}$ divides $f(1)-\tilde{f}(1)$. 
In both cases, Lemma \ref{elementary} provides  a 
holomorphic solution $\tau_2(0)\in\mathcal{O}_\Delta$.

Let $n\geq 1$. When $r\geq 2$, suppose that   
$\tau_1(k),\tau_2(k)\in\mathcal{O}_\Delta$,  $\tilde{f}(k+1)\in\mathbb{C}[t_2]_{\leq r-2}$ 
($0\leq k\leq n-1$)  have been constructed such that 
equation
\eqref{5.2} up to order $n-1$ in $z$ and the equation \eqref{5.3}
up to order $n$ in $z$ are satisfied. 
When $r\in \{ 0,1\}$,  suppose that $\tau_1(k),\tau_2(k)\in\mathcal{O}_\Delta$
($0\leq k\leq n-1$)  have been constructed such that 
equation
\eqref{5.2} up to order $n-1$ in $z$ and the equation \eqref{5.3}
up to order $n$ in $z$ are satisfied, with $\tilde{f}(k) =0$ ($1\leq k\leq n$).   
Then the coefficient of $z^{n}$ in  \eqref{5.2}
and the coefficient in $z^{n+1}$ of \eqref{5.3} give
for $\tau_1(n),\tau_2(n)$ and $\tilde{f}(n+1)$ the
equations 
\begin{align}\label{5.7}
&\frac{d}{dt_{2}} \tau_1(n)+   h_1(n) =0,\\
&\frac{d}{dt_{2}} (\tau_2(n)f(0)) +f(0)\frac{d}{dt_{2}}\tau_2(n)\nonumber\\
&+ \tau_1(n)(f(1)-\tilde{f}(1))
+ \tau_1(0)(f(n+1) - \tilde{f}(n+1)) +h_2 (n)=0,\label{5.8}
\end{align}
where $h_1(n),h_2(n)\in\mathcal{O}_\Delta$
are known. 
Let $\tau_{1}(n)$ be a solution of (\ref{5.7}). With this choice of $\tau_{1}(n)$, 
equation \eqref{5.8} in the unknown function $\tau_{2}(n)$ becomes 
\begin{equation}\label{5.9}
 \frac{d}{dt_{2}}(\tau_2(n)f(0)) +f(0)\frac{d}{dt_{2}} \tau_2(n)
-\tau_1(0)\tilde{f}(n+1) +h_3 (n)=0,
\end{equation}
where $h_3(n)\in\mathcal{O}_\Delta$ is  known.
Remark that (\ref{5.9}) is of the form (\ref{diff-eqn}). 
From Lemma \ref{elementary}, equation (\ref{5.9}), 
with any given $\tilde{f}(n+1)$,   has a solution,  which is holomorphic on $\Delta$,  when
$r\in \{ 0,1\}. $ We choose  $\tau_{2}(n)$ to be  a solution of (\ref{5.9}) with  $\tilde{f}(n+1):=0.$ 
When  $r\geq 2$ equation 
(\ref{5.9})  has a solution if and only if 
$t_2^{r-1}$ divides  $(\tau_1(0)\tilde{f}(n+1)-h_3(n))$.
We choose $\tilde{f}(n+1)\in\mathbb{C}[t_2]_{\leq r-2}$
to be the unique polynomial of degree at most $r-2$,  such that this property is satisfied,
and $\tau_{2}(n)\in \mathcal O_{\Delta}$ to be the unique solution 
of equation (\ref{5.9}) with this choice of $\tilde{f}(n+1).$ The first statement of claim v) is proved.\

It remains to prove the uniqueness of the normal form.
When $r\in\{0,1\}$ this follows from claim iii).
Suppose now that $r\geq 2$. Consider  two  normal
forms of type (\ref{n-f-u}), 
with functions $f$ and $\tilde{f}$.
From claim iii),   $f(0)=\tilde{f}(0)$.
Let  $\tau_1\in\mathbb{C}\{t_2,z]]^*$ and 
$\tau_2\in\mathbb{C}\{t_2,z]]$ which satisfy  
\eqref{5.2},  \eqref{5.3}. 
Going again through the above proof for  the existence
of the normal form, we find inductively that 
$\tau_1(n)\in\mathbb{C},\tau_2(n)=0$ and $f(n+1)=\tilde{f}(n+1)$
for any $n\geq 0$. The details are as in the proof of
Theorem \ref{t4.1} iii).
\end{proof}

Below a function $f\in \mathbb{C} \{ t_{2}, z\}$ is called associated to
a $(T)$-structure  over $\mathcal N_{2}$ if the $(T)$-structure is (holomorphically)   isomorphic to the $(T)$-structure 
$A_{1} = C_{1}$, $A_{2} = C_{2} + z f E.$  
From  Theorem \ref{t5.1} i), any  $(T)$-structure admits a (non-unique) associated function, which was obtained using
gauge isomorphisms.
In the next theorem  we will  exploit the additional
freedom provided by  isomorphisms which lift non-trivial automorphisms of $\mathcal N_{2}$, in order to
simplify the lower order terms of associated functions.

\begin{thm}\label{t5.2}
Consider an arbitrary $(T)$-structure  $(H, \nabla )$ over $\mathcal{N}_2$. 

i) The order $\textup{ord}_0f(0)\in\mathbb{Z}_{\geq 0}
\cup\{\infty\}$ of an associated function $f$ 
is a 
formal invariant of $(H, \nabla )$ 
(with $\textup{ord}_0f(0):=\infty$ when  $f(0)=0$).

ii) If some associated function $\tilde{f}$ of $(H, \nabla )$  
satisfies $\tilde{f}(0)=0$, then there is    
an associated function $f$ with
$f(0)=0$ and $f(1)=0$.

iii) If the order of an associated function $\tilde{f}$ of $(H, \nabla )$ 
is $r=\textup{ord}_0\tilde{f}(0)\in\mathbb{Z}_{\geq 0}$, then
there is  an associated function $f$ with $f(0)=t_2^r$.
\end{thm}

\begin{proof}
To prove claim i), we consider two $(T)$-structures, given by 
\begin{align}
\nonumber&A_1=C_1,A_2=C_2+zfE\\ 
\label{5.10}&\tilde{A_1}=C_1,\tilde{A_2}=C_2+z\tilde{f}E
\end{align}
with $f,\tilde{f}\in\mathbb{C}\{t_2,z\}$, 
an automorphism $h:(\mathbb{C}^2,0)\to(\mathbb{C}^2,0)$ of $\mathcal N_{2}$ 
and a matrix $\tilde{T}\in GL_2(\mathbb{C}\{t_2,z]])$ such that relation
\eqref{2.15} is satisfied. We will make relations in \eqref{2.15} explicit.
From Lemma \ref{aut-F-man} ii), $h$ is of the form
$h(t_1,t_2)=(t_1,\lambda (t_2 ))$ with
$\lambda\in t_{2}\mathbb{C}\{t_2\}^*$. 
We write
\begin{equation}\label{5.11}
\tilde{T}=\begin{pmatrix}a & \tilde{e}\\c&b\end{pmatrix}
\end{equation}
with $a,b,c,\tilde{e}\in\mathbb{C}\{t,z]]$.
Relations \eqref{2.15} become 
\begin{equation}\label{explicit12}
z\partial_j \tilde{T}+
\delta_{j1} \tilde{T} + 
\delta_{j2} \dot{\lambda}  ( A_2\circ\lambda )  \tilde{T}-\tilde{T}\tilde{A_j}=0, \ j\in \{ 1,2\} . 
\end{equation}
For $j=1$ relation  (\ref{explicit12})  gives  $a,b,c,\tilde{e}\in\mathbb{C}\{t_2,z]]$.
For $j=2$ it gives 
\begin{align}
0&= \begin{pmatrix}z\partial_2 a & z\partial_2\tilde{e} \nonumber\\
z\partial_2 c & z\partial_2b\end{pmatrix}
+ \begin{pmatrix}0&z\dot{\lambda} (f\circ\lambda )\\
\dot{\lambda} & 0 \end{pmatrix}
\begin{pmatrix}a & \tilde{e}\\c&b\end{pmatrix} 
- \begin{pmatrix}a & \tilde{e}\\c&b\end{pmatrix}
\begin{pmatrix}0&z\tilde{f}\\
1 & 0 \end{pmatrix} \nonumber\\ \label{5.12}
&= \begin{pmatrix}
z(\partial_2a + \dot{\lambda} c( f\circ\lambda ) -\tilde{e} &
z(\partial_2\tilde{e} + \dot{\lambda} b (f\circ\lambda ) 
-\tilde{f}a)\\
z\partial_2c+\dot{\lambda} a-b & 
z\partial_2b+\dot{\lambda}\tilde{ e}-z\tilde{f}c\end{pmatrix} .
\end{align}
The $(1,1)$-entry in the above matrix implies that $\tilde{e} (z)= z e(z)$ for 
$e\in\mathbb{C}\{t_2,z]]$. As $\tilde{T}$ is invertible, 
we deduce that  $a(0)(0),b(0)(0)\in\mathbb{C}^*$. 
The $(1,2)$ entry  in \eqref{5.12} gives 
\begin{equation}\label{5.13}
\dot{\lambda} b(0) ( f(0)\circ\lambda )=\tilde{f}(0)a(0),
\end{equation}
which  implies  $\textup{ord}_0f(0)=\textup{ord}_0\tilde{f}(0)$. 
Claim i) is proved.  \

To prove claims ii) and iii)  we start  with 
$\tilde{f}\in\mathbb{C}\{t_2,z\}$ 
as in the assumptions of these claims. We will find
$a,b,c,e, f\in \mathbb{C}\{t_2,z\}$ 
which satisfy  \eqref{5.12} (with $\tilde{e}:= ze$) 
and such that $f$ is in the form required  by these claims. 
In both cases $a\in\mathbb{C}\{t_2\}^*$ will be suitably chosen, independent on $z$,    
\begin{equation}\label{5.14}
b:=a^{-1}\in\mathbb{C}\{t_2\}^*,\ c:=0,\ 
e:=\dot{a}\in\mathbb{C}\{t_2\}
\end{equation}
and  $\lambda\in t_2\mathbb{C}\{t_2\}^*$ satisfies 
\begin{equation}\label{5.15}
\dot{\lambda}=a^{-2}.
\end{equation}
With these choices, three of the four relations  \eqref{5.12} are satisfied. The remaining relation is given
by the $(1,2)$ entry of the matrix and is equivalent to 
\begin{equation}\label{5.16}
z\ddot{a} + a^{-3} (f\circ\lambda )-\tilde{f}a=0.
\end{equation}
To prove claim ii),  
assume that $\tilde{f}(0) =0$ and choose $a$ such that 
\begin{equation}\label{5.17}
\ddot{a}= \tilde{f}(1)a.
\end{equation} 
Then \eqref{5.16} has a unique solution
$f\in\mathbb{C}\{t_2,z\}$ with $f(0)=f(1)=0$ and 
$f(k)\circ\lambda=\tilde{f}(k)a^4$ for $k\geq 2$.
Claim ii) is proved.\

To prove claim iii), 
assume that  $\mathrm{ord}_{0}\tilde{f}(0) =r\in \mathbb{Z}_{\geq 0}$.
We start with a solution 
$\lambda\in t_2\mathbb{C}\{t_2\}^*$ of the equation
\begin{equation}\label{5.18}
\lambda^r (\dot{\lambda})^2 = \tilde{f}(0),
\end{equation} 
(which exists from Lemma \ref{adaug-ajutatoare}).
Then we choose $a\in\mathbb{C}\{t_2\}^*$ such that \eqref{5.15}
holds, and then $b,c,e$ as in \eqref{5.14}.
The function $f\in\mathbb{C}\{t_2,z\}$ with
$f(0)=t_2^r$, $f(1)\circ\lambda=\tilde{f}(1)a^4
- a^{3} \ddot{a}$ 
and $f(k)\circ \lambda = \tilde{f}(k) a^{4}$,  for $k\geq 2$,  satisfies \eqref{5.16}.
Claim iii) is proved.
\end{proof}

\begin{rem}{\rm  The notion of an associated function can be extended to formal $(T)$-structures over $\mathcal N_{2}$, by replacing in their definition '(holomorphically) isomorphic' with 'formally isomorphic'.  For formal $(T)$-structures, associated functions belong to 
$\mathbb{C} \{ t_{2}, z]]$.  Theorem \ref{t5.1} ii) shows that any formal $(T)$-structure admits
an associated function. Theorem \ref{t5.2} holds also for formal $(T)$-structures.}
\end{rem}

Our main result from this section is the next theorem, which states the
formal classification of $(T)$-structures over $\mathcal N_{2}$.

\begin{thm}\label{t5.3}  i) Any  
$(T)$-structure  (or formal $(T)$-structure) 
over $\mathcal N_{2}$ is formally isomorphic  
to a   $(T)$-structure of the form
\begin{align}
\label{5.19a} &A_{1}= C_{1},\   A_{2}= C_{2} + z E\\
\label{5.19b} &A_{1}= C_{1},\  A_{2}= C_{2} + z t_{2} E\\
\label{5.19} & A_{1} = C_{1},\ A_{2} = C_{2},
\end{align}
or to a holomorphic or formal  $(T)$-structure of the form
\begin{equation}\label{5.20}
A_{1} = C_{1},\ A_{2} = C_{2} + z (t_{2}^{r} +\sum_{k\geq 1}
f(k)z^{k})E, 
\end{equation}
where $f(k)\in\mathbb{C}[t_2]_{\leq r-2}$ are polynomials 
of degree at most $r-2$ and  $r\in \mathbb{Z}_{\geq  2}$.\ 

ii)     Any two different  $(T)$-structures 
(or formal $(T)$-structures)  $\nabla$ and $\tilde{\nabla}$ from i), at least one of them not being of
the form (\ref{5.20}), are formally non-isomorphic. If $\nabla$ and $\tilde{\nabla}$ are of the form
(\ref{5.20}), then they are
formally gauge non-isomorphic. They are formally isomorphic if and only if there is $\lambda_{0}\in \mathbb{C}$, 
$\lambda_{0}^{r}=1$, such that $\tilde{f}(k) (t_{2}) = \lambda_{0}^{-2} f(k) (\frac{t_{2}}{\lambda_{0}})$,  
for any $k\geq 1$. 
\end{thm}

\begin{proof}
We only prove the statements for $(T)$-structures (the arguments for formal $(T)$-structures
are similar).  Let $\nabla $ be a  $(T)$-structure over $\mathcal N_{2}$ and $\tilde{f}$ an associated function.
If $\tilde{f}(0)\neq 0$ then, using  Theorem \ref{t5.2} iii), we can assume that $\tilde{f}(0) = t_{2}^{r}$ with $r\in \mathbb{Z}_{\geq 0}.$ 
Then Theorem \ref{t5.1} v)  implies that
$\nabla$ is formally isomorphic to a $(T)$-structure of the form 
(\ref{5.19a}), (\ref{5.19b}) or to a  $(T)$-structure  or formal $(T)$-structure of the form (\ref{5.20}). If $\tilde{f}(0) = 0$ then, using Theorem \ref{t5.2} ii),   we can assume that
$\tilde{f}(1)=0.$ Then Theorem  \ref{t5.1} iv)  implies that $\nabla$ is formally isomorphic to the $(T)$-structure (\ref{5.19}). 
 Claim i) is proved.\

The first two parts of claim ii) follow from Theorem \ref{t5.2} i) together with the uniqueness part
in Theorem \ref{t5.1} v).  Assume now that $\nabla$ and $\tilde{\nabla}$ are 
two formally isomorphic  $(T)$-structures of the form \eqref{5.20}. 
Let $T$ be a 
formal isomorphism  between them. It  covers a map  of the form $h(t_{1}, t_{2}) = (t_{1}, \lambda (t_{2}))$,
with $\lambda (0) =0$ and $\dot{\lambda}(0)\neq 0.$  
From relation (\ref{2.15}) together with  $A_{2}(0) = \tilde{A}_{2}(0) = C_{2}$ and 
$f(0) = \tilde{f}(0) = t_{2}^{r}$, we deduce that $\lambda$ satisfies $(\dot{\lambda})^{2} \lambda^{r} = t_{2}^{r}$.
From Lemma \ref{adaug-ajutatoare},  $\lambda (t_{2}) = \lambda_{0} t_{2}$, where $\lambda_{0}^{r} =1$. 
Consider now the isomorphism $T_{1}$ which covers $h $ and is given by the
constant matrix $\mathrm{diag} (1, \lambda_{0}).$ Then $\nabla^{(1)}:= T_{1}\cdot \nabla$ is a $(T)$-structure with 
$A_{1}^{(1)} = C_{1}$, $A_{2}^{(1)} = C_{2} + z f^{(1)}E$, where $f^{(1)}(0) = t_{2}^{r}$ and $f^{(1)}(k) (t_{2}) = \lambda_{0}^{-2} f(k) (\frac{t_{2}}{\lambda_{0}})$, for any
$k\geq 1.$ 
Remark that $\nabla^{(1)}$ and $\tilde{\nabla}$ are formally gauge isomorphic (by means of the formal gauge isomorphism
$T\circ T_{1}^{-1}$). Therefore, they coincide. We deduce that $\tilde{f}= f^{(1)}$, which completes the proof of  
claim ii). 
\end{proof}

\begin{rem}\label{t5.4}{\rm
It is natural to ask if a $(T)$-structure is holomorphically isomorphic to its formal normal form provided by
Theorems \ref{t4.1} iii)  or \ref{t5.3}. 
We believe that this is not, in general, true. 
Let us consider the $(T)$-structures over $I_{2}(m)$, 
with $m\geq 4.$ 
We do not believe that the functions 
$\tau_1,\tau_2$ and $\tilde{f}$ 
constructed in the proof of part iii) of Theorem \ref{t4.1} 
are in general holomorphic if $f$ is holomorphic.
Therefore, in order to obtain holomorphic normal forms for $(T)$-structures over $I_{2}(m)$ one needs to modify 
part iii) in Theorem \ref{t4.1} substantially. 
We plan to work on this.
We also plan to work on the formal and holomorphic classification
of $(TE)$-structures.   }
\end{rem}

\end{document}